\documentclass[12pt]{article}
\usepackage{amsmath, amssymb, amsthm, tikz, subcaption, longtable, array, listings, mathtools, enumitem, lineno, setspace, adjustbox}
\usetikzlibrary{positioning,fit,calc,decorations.pathreplacing}
\usetikzlibrary{fit,backgrounds}
\usepackage[colorlinks = true]{hyperref}
\usepackage[margin = 2 cm]{geometry}
\usepackage{tocloft}

\newcolumntype{P}[1]{>{\centering\arraybackslash}p{#1}}

\title{Structural Characterizations and Algebraic Realizations of a Family of Regular Integral Graphs}
\author{Tapa Manna$^a$\thanks{Email: \texttt{mannatapa24@gmail.com}}, Supriyo Dutta$^a$\thanks{Email: \texttt{dosupriyo@gmail.com}}, Baby Bhattacharya$^a$\thanks{Email: \texttt{babybhatt75@gmail.com}}  \\ 
	$^a$ \small{Department of Mathematics, National Institute of Technology Agartala, Jirania, Tripura, India - 799046}\\ }
\date{}

\newtheorem{lemma}{\bf Lemma}[section]

\newtheorem{theorem}{\bf Theorem}[section]
\newtheorem{proposition}{\bf Proposition}[section]
\newtheorem{corollary}{\bf Corollary}[section]
\newtheorem{definition}{\bf Definition}[section]

\newtheorem{remark}{\bf Remark}[section]

\newtheorem{note}{\bf Note}[section]
\newtheorem{procedure}{\bf Procedure}[section]

\DeclareMathOperator{\diag}{diag}

\begin{document}
	
	\maketitle
	
	\begin{abstract}
		All the eigenvalues of an integral graphs are integers. Integral graphs are extremely rare. They form an asymptotically vanishing fraction $2^{-\Omega(n)}$ among all graphs on $n$ vertices. It makes the construction of a new family of integral graphs a challenging task. Also, most of the known infinite family of integral graphs rely on Cayley graphs over Abelian groups. In this article, we introduce a new family of integral graphs obtained from the groups. The construction of our graphs from groups is different from the construction of Cayley graphs. A spectral uniqueness theorem is established, which shows that each member of the infinite family is determined by its adjacency spectrum among all finite simple graphs. We also present recursive constructions that generates larger members of the family from smaller ones, providing a scalable class of integral graphs. Finally, we investigate algebraic realizations of these graphs as complements of Proper Prime Order Element Graphs of finite $2$-groups and obtain conditions characterizing such realizations. We also observe that the graphs obtained from different non-isomorphic groups have cospectral graphs.
		
	\end{abstract}

	\tableofcontents
	
	\section{Introduction}

	A graph is an integral graph if all its eigenvalues are integers. Constructing a new family of integral graphs is a challenging problem because there is no necessary and sufficient condition to characterize the integral graphs till now. It makes enlisting all integral graphs or classifying the graphs with integral spectra is difficult. The problem of finding integral graphs was first introduced in \cite{harary2006graphs}. In the recent years we observe a number of important developments in this direction \cite{bali2002nska,bussemaker1976there,stevanovic20034,stevanovic2007walks,balinska2001nonregular,so2006integral}. In this article, we illustrate a new family of graphs with integral spectra.

	There is no concrete rule to categorize graphs with integral spectra. The literature contains a plenty of methods to construct these graphs. These graphs form a restrictive but structurally rich class within spectral graph theory. Article \cite{roitman1984infinite} constructed an infinite family of integral graphs. Article \cite{wang2004integral} demonstrated the complete integral multipartite graphs. Further constructions were obtained in \cite{hic2008new}, and \cite{wang2008integral}. Integral trees are discussed in  \cite{brouwer2007integral,brouwer2008small}. The Cayley graphs act an important role in constructing integral graphs. Fundamentally, they build up a bridge between algebra and combinatorics. The integral Cayley graphs over the finite Abelian groups are discussed in \cite{klotz2010integral}. The connection between the algebraic structure of a group and the integrality of its Cayley graphs was developed by \cite{alperin2012integral}. These works suggest the importance of groups in the construction of the integral graphs \cite{abdollahi2014groups,ahmady2014integral}. In this article, we establish another connection that allows  us to construct a new family of integral graphs using the ideas from group theory. To the best of our knowledge, this is the first alternative of the idea of Cayley graphs when we construct integral graphs from the finite groups. Apart from the Cayley graphs, there are many other approaches for constructing graphs from groups, which provide important structural characteristics of the group. Important examples of graphs defined on groups include the commuting graph, power graph, enhanced power graph, generating graph, and their variants \cite{chakrabarty2009undirected,cameron2010power, aalipour2016structure,arunkumar2022super,cameron2021graphs}.
	 Prime Order Element Graphs have recently been investigated from both structural and spectral viewpoints \cite{manna2025prime,manna2024forbidden,manna2026structural}. Motivated by these studies, we consider the complement of the Proper Prime Order Element Graph of a finite group. We denote this graph by $D^*(G)$. Our main interest lies in finite $2$-groups, where the adjacency condition takes a particularly simple form and leads naturally to an infinite two-parameter family of connected regular integral graphs.
	 
	 The second part of the paper addresses the corresponding realization problem with a group-theoretic approach. Rather than merely listing examples of realizing groups, we obtain an exact characterization of those finite $2$-groups $G$ for which $D^*(G)\cong X_{h,r}.$ We then study the non-uniqueness of such realizations. In particular,
	 we exhibit non-isomorphic Abelian and non-Abelian finite $2$-groups that give rise to the same graph $X_{h,r}$ and hence to the same adjacency spectrum. Thus, although $X_{h,r}$ is determined by its adjacency spectrum, its underlying group realization need not be unique.
	
	In \autoref{forward_realization}, we prove that the graph $D^*(G)$ is isomorphic to a two parameter family of graphs, denoted by $X_{h,r}$, which has a
	structured join decomposition consisting of an independent set and $r$ number of disjoint $K_2$ graphs. Conversely, we investigate the realization problem of determining
	which members of the family $X_{h,r}$ arise as $D^*(G)$ for finite $2$-groups $G$. Now, we have the following results:
	
	\begin{enumerate}
		\item The graph $X_{h,r}$ is an integral graph with  $\Lambda (X_{h,r})
		=\left\{[rh]^1, [1-h]^r, [1]^{r\left(\frac{h}{2}-1\right)},\,
		[0]^{h-2},[-1]^{\frac{rh}{2}} \right\}.$
		
		\item Let $G$ be a finite $2$-group such that $D^*(G)\cong X_{h,r}.$ Let $n=|V(D^*(G))|=|G|-1$ and let $\rho$ denote the adjacency spectral radius of the adjacency matrix of $D^*(G)$. Then the adjacency spectrum of $D^*(G)$ determines the group-theoretic quantities $|G|$, $r_2(G)$, $|H(G)|$, and $[G:H(G)]$, where $H(G)=\{x\in G:x^2=e\}.$ More precisely, $|G|=n+1, |H(G)|=n-\rho+1, r_2(G)=n-\rho,$ and $[G:H(G)]=1+\frac{\rho}{\,n-\rho+1\,}.$
		
		\item Let $D^{*}(G)\cong X_{h,r}$.Then there exist positive integers $m$ and $d$ such that $h=2^{m} \text{ and } r=2^{d}-1.$ Equivalently, every graph in the family $X_{h,r}$ that is realizable as $D^{*}(G)$ has the form $X_{2^{m},\,2^{d}-1}.$
		
		\item For every integer $m\geq3$, the graph $X_{2^m,\,1}$
		admits both an Abelian and a non-Abelian realization. In particular, $D^{*}\left(
		\mathbb Z_4\times\mathbb Z_2^{\,m-1} \right) \cong
		D^{*}(G_{m,\theta}) \cong X_{2^m,\,1}.$
		
		\item The graph $X_{h,r}$ is determined by its adjacency spectrum.
	\end{enumerate}
	
	A further significance of the family lies in the contrast between spectral rigidity and algebraic non-uniqueness. Each graph $X_{h,r}$ is determined by its adjacency spectrum among all finite simple graphs; in other words, $X_{h,r}$ is  determined by its adjacency spectrum. On the other hand, non-isomorphic finite $2$-groups may realize the same graph $X_{h,r}$ and therefore have identical associated adjacency spectra. Thus, the family provides an algebraically motivated class of graphs that are determined by their adjacency spectra and connects the present group-theoretic construction with the spectral characterization problem studied in \cite{van2003graphs}.

	For standard terminology and results from group theory, linear algebra, and graph theory, we refer to \cite{dummitbasic}, \cite{friedberg1997linear}, \cite{west2001introduction}. For elementary definitions we recommend to check the referred textbooks.

	\section{Preliminaries}
	
	\subsection*{Group-Theory Essentials} \label{subsec:group-terminology}

	 Given a subgroup $H$ of a group $G$ the left(or right) cosets with cardinality $|H|$ are defined by $gH=\{gh: h\in H\}$ (or $Hg=\{hg: h\in H\}$) for some $g\in G$. A subgroup $H$ is a normal subgroup of $G$ if $gH=Hg \forall g\in G$, and we denote $H\trianglelefteq G$. The number of distinct left cosets of $H$ in $G$ is called the \emph{index} of $H$ in $G$ and is denoted by $ [G:H].$ For a finite group, $[G:H]=\frac{|G|}{|H|}.$ A subgroup $H\leq G$ is called \emph{characteristic} in $G$ if $\varphi(H)=H$ for every automorphism $\varphi\in\operatorname{Aut}(G).$ The quotient group of $G$  by the normal subgroup $H$ is $G/H=\{gH:g\in G\}$ under the operation $(gH)(kH)=gkH$, with order $|G/H|=[G:H].$

	 A group $G$ is called a $2$-group if $|G|=2^n$ for $n\in \mathbb{N}$. Order of any element in a $2$-group is a power of $2$. The \emph{exponent} of $G$, denoted by $\exp(G),$ is the least positive integer $m$ such that $g^m=e$ for every $g\in G$. An element $g\in G$ is called an \emph{involution} if $g\neq e, \text{ the identity element, and } g^2=e.$ The set of involutions of $G$ is $I(G)=\{x\in G:x\neq e,\ x^2=e\}$. Denote $r_2(G)=|I(G)|$. The set $H(G)=I(G)\cup \{e\}$, need not be a subgroup of $G$. Whenever $H(G)$ is a subgroup, we state the fact explicitly.

	 The \emph{center} of $G$ is the set $Z(G)=\{z\in G : zg=gz \text{ for every } g\in G\}.$

	 The \emph{direct product} of two groups $G$ and $K$, denoted by $ G\times K, $ and is defined by $\{(g,k):g\in G,\ k\in K\}$ where $ (g_1,k_1)(g_2,k_2) = (g_1g_2,k_1k_2).$ A group $E$ is called an \emph{elementary Abelian $2$-group} if $E\cong\mathbb{Z}_2 \times \mathbb{Z}_2\times \cdots \mathbb{Z}_2$ for some $m\in \mathbb{N}$.

	  An action of a group $K$ on a group $N$ by automorphisms is a homomorphism $ \theta:K\longrightarrow\operatorname{Aut}(N).$ For $k\in K$ and $n\in N$, the image of $n$ under the automorphism $\theta(k)$ is denoted by $ \theta(k)(n).$ Let $N$ and $K$ be groups, and let $\theta:K\longrightarrow\operatorname{Aut}(N)$ be a homomorphism. The \emph{semidirect product} of $N$ by $K$ with respect to $\theta$, denoted by $ N\rtimes_{\theta}K,$ is the group whose underlying set is $N\times K$ and whose operation is $ (n_1,k_1)(n_2,k_2) = \bigl(n_1\theta(k_1)(n_2),k_1k_2\bigr).$ When the action $\theta$ is clear from the context, we simply write $N\rtimes K.$ 
	  
	  The quaternion group $Q_8$ is a non-Abelian group of order $8$ given by $ Q_8 = \langle a,b: a^4=e,\ a^2=b^2,\ b^{-1}ab=a^{-1}\rangle. $ Its unique involution is $-1$ that is $a^2=b^2$.
	  
	  For two disjoint sets $A$ and $B$, we write
	  $A\dot\cup B$ for their disjoint union.

	\subsection*{Graph-Theory Essentials}\label{subsec:graph-terminology}
	
	A finite simple undirected graph is a pair $\mathcal{G}=(V(\mathcal{G}),E(\mathcal{G})),$ where $V(\mathcal{G})$ is the \emph{vertex set}, and $ E(\mathcal{G}) \subseteq \bigl\{(u,v):u,v\in V(\mathcal{G}),\ u\neq v\bigr\}$ is the \emph{edge set}. Let $V(\mathcal{G})=\{v_1,v_2,\cdots, v_n\}$ Two vertices $u$ and $v$ in $V(\mathcal{G})$ are adjacent if an edge $(u,v)\in E(\mathcal{G})$. The \emph{adjacency matrix} of $\mathcal{G}$ is the $n\times n$ matrix $A(\mathcal{G})=(a_{ij}),$ where 
	\begin{equation}\label{adjacency_equation}
		a_{ij} = \begin{cases} 1,&\text{if }v_i\sim_{\mathcal{G}}v_j,\\ 0,&\text{otherwise}. \end{cases}
	\end{equation} The \emph{adjacency spectrum} of $\mathcal{G}$ is the multiset of all eigenvalues of $A(\mathcal{G})$, counted according to their algebraic multiplicities. If the distinct eigenvalues of $\mathcal{G}$ are $ \lambda_1,\lambda_2,\ldots,\lambda_s$ with respective multiplicities $m_1,m_2,\ldots,m_s,$ then we write $\Lambda (\mathcal{G}) = \left\{ [\lambda_1]^{m_1}, [\lambda_2]^{m_2}, \ldots, [\lambda_s]^{m_s} \right\}.$ Recall that $m_1+m_2+\cdots +m_s=|V(\mathcal(G))|$.
	
	The degree of a vertex $v$ of a graph $\mathcal{G}$, denoted by $d^{\mathcal{G}}_v$ is the number of vertices adjacent to it. \emph{Degree Matrix} of a graph $\mathcal{G}$ is $D(\mathcal{G})=\diag \{d_{v_1},d_{v_2},\cdots,d_{v_n}\}$, where $d_{v_i}$ is the degree of the vertex $v_i$. Also, the Laplacian matrix is $L(\mathcal{G})=D(\mathcal{G})-A(\mathcal{G})$. Two graphs are called \emph{cospectral} if they have the same
	adjacency spectrum. A graph $G$ is said to be \emph{determined by its adjacency spectrum}, or briefly \emph{adjacency-DS}, if every graph cospectral with $G$ is isomorphic to $G$.
	
	 A graph $\mathcal{G}$ is called \emph{$k$-regular} if $ d^{\mathcal{G}}_v=k$ for every $v\in V(\mathcal{G})$. The graph with $n$ vertices is complete graph $K_n$ where every pair of distinct vertices is adjacent. A graph with $n$ vertices and no edges is called the \emph{null graph}, which is denoted by $\overline{K}_n$. 
	 
	 A subset $S\subseteq V(\mathcal{G})$ is called an \emph{independent set} if no two distinct vertices of $S$ are adjacent in $\mathcal{G}$. The \emph{subgraph of $\mathcal{G}$ induced by $S\subseteq V(\mathcal{G})$}, denoted by $\mathcal{G}[S]$, is the graph with vertex set $S$ and edge set $ E(\mathcal{G}[S]) = \bigl\{(u,v)\in E(\mathcal{G}):u,v\in S\bigr\}.$ The \emph{complement} of $\mathcal{G}$, denoted by $\overline{\mathcal{G}}$, is the graph with $ V(\overline{\mathcal{G}})=V(\mathcal{G})$ such that, for distinct vertices $u,v\in V(\mathcal{G})$, $ u\sim_{\overline{\mathcal{G}}}v \quad\Longleftrightarrow\quad u\not\sim_{\mathcal{G}}v.$ Equivalently, $ E(\overline{\mathcal{G}}) = E(K_{|V(\mathcal{G})|})\setminus E(\mathcal{G}).$

	 The \emph{join} of $\mathcal{G}_1$ and $\mathcal{G}_2$, denoted by $ \mathcal{G}_1\vee\mathcal{G}_2,$ has the vertex set $V(\mathcal{G}_1\cup\mathcal{G}_2)$ and edge set $ E(\mathcal{G}_1\vee\mathcal{G}_2) = E(\mathcal{G}_1)\cup E(\mathcal{G}_2) \cup \bigl\{(u,v):u\in V(\mathcal{G}_1), v\in V(\mathcal{G}_2)\bigr\}.$ More generally, for pairwise vertex-disjoint graphs $\mathcal{G}_1,\ldots,\mathcal{G}_r$, the notation $\bigvee_{i=1}^{r}\mathcal{G}_i$ denotes their successive join. 
	 
	  Two graphs $\mathcal{G}_1$ and $\mathcal{G}_2$ are said to be isomorphic if there is a bijection, in other words a graph isomorphism $\varphi:V(\mathcal{G}_1)\longrightarrow V(\mathcal{G}_2)$ such that $ u\sim_{\mathcal{G}_1}v \Longleftrightarrow \varphi(u)\sim_{\mathcal{G}_2}\varphi(v)$ for every pair of distinct vertices $u,v\in V(\mathcal{G}_1)$. We write $\mathcal(G)_1\cong \mathcal(G)_2$.

	 \begin{lemma}\label{Rayleigh_quotient}\cite{friedberg1997linear}
	 	Let $A$ be a real symmetric matrix with eigenvalues
	 	$\lambda_1\geq \lambda_2\geq\cdots\geq\lambda_n.$
	 	Then $\lambda_n \leq \frac{x^{T}Ax}{x^{T}x}\leq\lambda_1$
	 	for every nonzero vector $x\in\mathbb{R}^n$. Moreover, $\lambda_1 = \max_{x\neq 0} \frac{x^{T}Ax}{x^{T}x},\quad
	 	\lambda_n = \min_{x\neq 0} \frac{x^{T}Ax}{x^{T}x}.$
	 	Equality holds if and only if $x$ is an eigenvector corresponding to the respective extremal eigenvalue.
	 \end{lemma}
	 
	 \begin{lemma}\label{lem:laplacian-connectedness}\cite{west2001introduction}
	 	A finite simple graph $\mathcal{G}$ is connected iff its Laplacian matrix has exactly one zero eigenvalue.
	 \end{lemma}
	 
	 \begin{lemma}\label{lem:complement-spectrum}\cite{west2001introduction}
	 	Let $\mathcal{G}$ be a $k$-regular graph on $n$ vertices with adjacency spectrum $\Lambda(\mathcal{G})=
	 	\{[k]^1,\lambda_2,\ldots,\lambda_n\}.$ Then the complement $\overline{\mathcal{G}}$ is $(n-k-1)$-regular and $\Lambda (\overline{\mathcal{G}})
	 	=\{[n-k-1]^1,-1-\lambda_2,\ldots,-1-\lambda_n\}.$
	 	
	 \end{lemma}

	\section{The Proper POE Complement Graphs and their properties}
	
	We now introduce the graph construction that forms the algebraic basis of our study. We first recall the Proper Prime Order Element Graph $\Gamma^*(G)$ introduced from Prime Order Element Graph in \cite{manna2025prime}. We then consider its complement and denote the resulting graph by $D^*(G)$. This graph will play a central role in the subsequent structural and spectral investigations.
	
		\begin{definition}
			 Let $G$ be a finite group with identity element $e$, and let $G^*=G\setminus\{e\}$. The Proper Prime Order Element Graph $\Gamma^*(G)$ is the graph with vertex set $G^*$ in which two distinct vertices $x,y\in G^*$ are adjacent if and only if $o(xy)$ is prime.
		\end{definition}

\begin{definition}\label{complement_POE}
	Let $G$ be a finite group. The \emph{Proper POE Complement Graph} of $G$, denoted by $D^*(G)$, is defined by
	$D^*(G)=\overline{\Gamma^*(G)}.$ Thus, $V(D^*(G))=G^*,$ and two distinct vertices $x,y\in G^*$ are adjacent in $D^*(G)$
	if and only if $o(xy)\text{ is not prime}.$ 
\end{definition}

 \begin{definition}\label{realizable_graph}
 	 We define a graph $\mathcal{G}$ is
 	\emph{realizable as a Proper Order POE Complement Graph} if there exists a finite group $G$ such that $D^{*}(G)\cong \mathcal{G}.$ The group $G$ is called a \emph{group realization} of the graph $\mathcal{G}$.
 \end{definition}
 
\begin{lemma}\label{adjacent_lemma}
	Let $G$ be a finite $2$-group. For two distinct vertices
	$x,y\in G^*$, $x\sim_{D^*(G)}y \iff o(xy)\neq 2.$
\end{lemma}

\begin{proof}
	Since $G$ is a finite $2$-group, the order of every element of $G$ is a power of $2$. Hence the only possible prime order of $xy$ is $2$. By the definition of $D^*(G)$,
	$x\sim_{D^*(G)}y \iff o(xy)\text{ is not prime},$ and therefore
	$x\sim_{D^*(G)}y\iff o(xy)\neq2.$
\end{proof}
 
 \begin{theorem}\label{degree_theorem}
 	Let $r_2(G)$ denote the number of involutions in a finite $2$-group $G$. Then, for every $x\in G^*$,
 	\[
 	\deg_{D^*(G)}(x)=
 	\begin{cases}
 		|G|-1-r_2(G), & \text{if }o(x)=2\text{ or }4,\\[2mm]
 		|G|-2-r_2(G), & \text{if }o(x)\geq8.
 	\end{cases}
 	\]
 	
 \end{theorem}
 \begin{proof}
 	By the \autoref{adjacent_lemma}, two distinct non-identity elements $x$ and $y$ are adjacent in $D^*(G)$ when $o(xy)\neq 2$. Thus, to determine the degree of $x$ in $D^*(G)$, we count
 	the vertices that are non-adjacent to $x$ in $\Gamma^*(G)$.
 	Equivalently, these are the vertices adjacent to $x$ in
 	$\overline{\Gamma^*(G)}$, namely those $y\in G^*$ for which
 	$o(xy)\neq 2$. Let $I_G$ be the set of involutions in $G$. Then, $|I_G|=r_2(G)$. For every involution $t\in I_G$, the equation $xy=t$ has exactly one solution, that is, $y=x^{-1}t$. Therefore, initially, there are $r_2(G)$ possible neighbors of $x$ in $\Gamma^*(G)$. Also, we need to consider the case of $y=e$ and $y=x$ further. Hence, we have the following cases. 
 	
 	\textit{Case-I:} If $o(x)=2$, that is if $x$ itself is an involution. Take $t=x$, then, $y=x^{-1}t=x^{-1}x=e$. But $e$ is not the vertex of the graph $D^*(G)$. Therefore, one of the $r_2(G)$ solutions of the equation $xy=t$ is unacceptable. It indicates that the number of neighbors of $x$ in $\Gamma^*(G)$ is $(r_2(G)-1)$. \autoref{complement_POE} indicates that $D^*(G)=\overline{\Gamma^*(G)}$. Since, $x\in G^*$ cannot be adjacent to $x$ and $x^{-1}$ in $\Gamma^*(G)$, the degree of $x$ in $D^*(G)$ is $\deg(D^*(G)(x)) = (|G|-2)-(r_2(G)-1) = |G|-1-r_2(G)$. 
 	
 	\textit{Case-II:} If $o(x)=4$, that is if $x^2$ is an involution. Take $t=x^2$, then $y=x^{-1}t=x^{-1}x^2=x$. But we cannot consider $y=x$ because $\Gamma^*(G)$ is a simple graph having no loop. Hence again, one of the $r_2(G)$ solutions of the equation $xy=t$ is unacceptable. In this case, the number of neighbors of $x$ in $\Gamma^*(G)$ is $r_2(G)-1$. Therefore, $\deg(D^*(G)(x))=|G|-1-r_2(G)$.
 	
 	\textit{Case-III:} If $o(x)\geq 8$, then neither $x$ nor $x^2$ is an involution. Thus, no involution $t$ provides $y=e$ or $y=x$. Therefore, all $r_2(G)$ solutions of the equation $xy=t$ represents neighbors in $\Gamma^*(G)$. Hence, $\deg(D^*(G)(x))=|G|-2-r_2(G)$.
 \end{proof}
 
 \begin{corollary}\label{exponent_corollary}
 	Let $G$ be a finite $2$-group with the number of involutions $r_2(G)$. Then $D^*(G)$ is regular if and only if $exp(G)\leq 4$.
 \end{corollary}
 \begin{proof}
 	Let $exp(G)\leq 4$, then every non-identity element $x\in G$ has order either $2$ or $4$. By \autoref{degree_theorem}, for every vertex $x\in G^*$, we have $\deg(D^*(G)(x))=|G|-1-r_2(G)$. Thus, $D^*(G)$ is a regular graph.
 	
 	Conversely, let $D^*(G)$ be a regular graph. Suppose, to the contrary, that, assume $exp(G)\geq 8$. Then there exists an element $x\in G$ such that $o(x)\geq 8$. Since every finite $2$-group contains an involution, we take $z\in G$ such that $o(z)=2$. Following \autoref{degree_theorem}, $\deg_{D^*(G)}(x) = |G|-2-r_2(G) \neq \deg_{D^*(G)}(z) = |G|-1-r_2(G)$. It contradicts the regularity of $D^*(G)$. Hence $exp(G)\leq 4$. 
 \end{proof}

	\section{The Graph Family $X_{h,r}$: Structure and Spectrum}
 	
		In \autoref{subsec:graph-terminology}, we define the join of two graphs. In this section, we use this idea for constructing the integral graphs. We have already mentioned that there are groups $G$ corresponding to these graphs such that $D^*(G)\cong X_{h,r}$ for the two parameters $h$ and $r$. 
 
		\begin{definition}\label{infinite_graph_family}
			Let $h\geq 2$ be an even integer and let $r\geq 1$. Define $X_{h,r}
		 	=
		 	\overline{K}_{h-1}
		 	\vee
		 	\left(
		 	\bigvee_{i=1}^{r} \left(\frac{h}{2}\right)K_2
		 	\right),$ where $\overline{K}_{h-1}$ is the null graph on $h-1$ vertices and
		 	$\left(\frac{h}{2}\right)K_2$ is the disjoint union of $\frac{h}{2}$ copies of $K_2$ graphs. 
		\end{definition}
		From now onward, we assume $h,r\in \mathbb{Z}$ such that $h$ is even with $h\geq 2$ and $r\geq 1$. \autoref{all-scalable-family} illustrates the examples for $h=2,4,8,16$ and $r=1,2,3,4$. Now we discuss a number of characteristics of these graphs. 

		\begin{lemma}\label{X_h_r_basic}
			The graph $X_{h,r}$ is a connected $rh$-regular graph with $(r+1)h-1$ vertices. Moreover, the complement of $X_{h,r}$ is isomorphic to $K_{h-1}\,\dot{\cup}\,
			r\left(K_h-\left(\frac{h}{2}\right)K_2\right).$
		\end{lemma}
		\begin{proof}
			Note that the vertex set of $X_{h,r}$ admits a partition $V(X_{h,r}) = V_0\dot\cup V_1\dot\cup\cdots\dot\cup V_r$, such that $X_{h,r}[V_0]\cong \overline{K}_{h-1}$, $X_{h,r}[V_i] \cong \left(\frac{h}{2}\right)K_2$ for $1 \leq i \leq r$, The graph $\overline{K}_{h-1}$ has $h-1$ vertices, while each of the
			$r$ copies of $\left(\frac{h}{2}\right)K_2$ has $h$ vertices. Hence the number of vertices in $X_{h,r}$ is $(h-1) + rh =(r+1)h-1$. 
			
			Due to the structure generated by graph join, any two vertices belonging to distinct vertex-partitions are adjacent. Thus, the vertex $v\in V_0$ has no neighbor in $V_0$. But it is adjacent to every vertex in $V_1\cup\cdots\cup V_r$. Therefore, $\deg(v)=rh.$ If $v\in V_i$ for some $1\leq i\leq r$, then $v$ has exactly one neighbor in $V_i$. In addition, $v$ is adjacent to all $h-1$ vertices of $V_0$ and to all the vertices in the $r-1$ vertex partitions $V_j$ for $j = 1, 2, \dots r$, $j \neq i$. Hence, $\deg(v)
			= 1+(h-1)+(r-1)h = rh.$ It indicates that, $X_{h,r}$ is an $rh$-regular graph. 
			
			For $r\geq 1$, the graph $X_{h,r}$ is the join of at least two non-empty graphs, every vertex in one part is adjacent to every vertex in each of the other parts. Therefore, $X_{h,r}$ is connected. 
			
			Finally, taking complements converts the join into a disjoint union. Moreover, $\overline{\overline{K}_{h-1}}=K_{h-1}$ and $\overline{\left(\frac{h}{2}\right)K_2} = K_h - \left(\frac{h}{2}\right)K_2.$ Thus, $\overline{X_{h,r}} \cong K_{h-1}\,\dot{\cup}\, r\left(K_h-\left(\frac{h}{2}\right)K_2\right).$ 
	 	\end{proof}
	 	
	 	The graph $K_{h-1}\dot\cup r\left(K_n-\left(\frac{h}{2}\right)K_2\right)$ is a disconnected graph. Its connected components are isomorphic to $K_{h-1}$ and $K_h-\left(\frac{h}{2}\right)K_2$. Moreover, the graph $K_h-\left(\frac{h}{2}\right)K_2$ is obtained from $K_h$ by deleting one edge incident with each vertex and that is why it is regular with degree $h-2$.

		\begin{lemma}\label{cocktail_lemma}
			The graph $K_h-\left(\frac{h}{2}\right)K_2$ is integral with adjacency spectrum 
			$$\Lambda\left(K_h-\left(\frac{h}{2}\right)K_2\right) = \left\{ [h-2]^1,\, [0]^{h/2},\,[-2]^{h/2-1} \right\}.$$ 
		\end{lemma}
		\begin{proof}
		 	Note that, $K_h$ and $\left(\frac{h}{2}\right)K_2$ have the same set of vertices. Also, $$\left(\frac{h}{2}\right)K_2
		 	=
		 	\{(u_1,v_1),(u_2,v_2),\ldots,(u_{h/2},v_{h/2})\}.$$ The graph $K_h-\left(\frac{h}{2}\right)K_2$ is obtained from $K_h$ by deleting one edge
		 	incident with each vertex. Thus, every vertex has degree $h-2$, and $K_h-\left(\frac{h}{2}\right)K_2$ is a regular graph, with regularity $(h-2)$. Hence, $h-2$ is an eigenvalue with all-ones eigenvector $\mathbf{1}$.
		 	
		 	Let $e_u$ be a vector corresponding to the vertex $u$ such that $e_u=(0,0,0,\cdots,1(\text{ u-th position }),\cdots, 0)^t$ of length $h$. For each $1\leq i\leq h/2$, consider the vector $(e_{u_i}-e_{v_i}).$ The vertices $u_i$ and $v_i$ are non-adjacent and have same neighbourhood in $K_h-\left(\frac{h}{2}\right)K_2$. Note that, $A\left(e_{u_i}-e_{v_i}\right)=0.$ There are $\frac{h}{2}$ vectors $(e_{u_i}-e_{v_i})$ which are linearly independent. Hence, $0$ is an eigenvalue with multiplicity $\frac{h}{2}$. 
		 	
		 	Next, we consider vectors that are constant on each pair $\{u_i,v_i\}$ and whose total coordinate sum is zero. This subspace has dimension $\left(\frac{h}{2}-1\right)$. Any vector $w$ in this subspace has the form the form $(\alpha_1,\alpha_1,\alpha_2,\alpha_2,\cdots , \alpha_{\frac{h}{2}},\alpha_{\frac{h}{2}})$, where $2(\alpha_1+\alpha_2+\cdots+\alpha_{\frac{h}{2}})=0$. The adjacency matrix of $K_h-\left(\frac{h}{2}\right)K_2$ is $$\begin{pmatrix}
		 		O_{2\times 2} & J_{2\times 2} & \cdots & J_{2\times 2}\\
		 		J_{2\times 2} & O_{2\times 2} & \cdots & J_{2\times 2}\\
		 		\cdots & \cdots & \cdots & \cdots \\
		 		J_{2\times 2} & J_{2\times 2} & \cdots & O_{2\times 2}\\
		 	\end{pmatrix}$$, where $O_{2\times 2}$ is the zero matrix of order $2$ and $J_{2\times 2}$ is the all one matrix of order $2$. Therefore, $Aw=-2w$. Thus, $-2$ is an eigenvalue with multiplicity at least $\left(\frac{h}{2}-1\right)$. 
		 	
		 	The total multiplicity corresponding to all the eigenvalues is $1+\frac{h}{2}+\left(\frac{h}{2}-1\right)=h,$ which is equal to the number of vertices. Hence $\Lambda\left(K_h-\left(\frac{h}{2}\right)K_2\right)
		 	=
		 	\left\{
		 	[h-2]^1,\,
		 	[0]^{h/2},\,
		 	[-2]^{h/2-1}
		 	\right\}.$
		 \end{proof}
		 
		 	\begin{note}
		 		Putting $h=2$ in \autoref{cocktail_lemma}, the graph is isomorphic to $K_2-\left(\frac{2}{2}\right)K_2$ which is two isolated vertices. Therefore, displayed spectrum will be $\{[2-2]^1,[0]^1\}$ or $\{[0]^2\}$.
		 	\end{note}

 \begin{theorem}\label{Xhr_spectrum}
 	The graph $X_{h,r}$ is an integral graph with  \[\Lambda(X_{h,r})
 	=
 	\left\{
 	[rh]^1, [1-h]^r, [1]^{r\left(\frac{h}{2}-1\right)},\,
 	[0]^{h-2},[-1]^{\frac{rh}{2}} \right\}.\]  
 \end{theorem}
 \begin{proof}
 	By \autoref{X_h_r_basic}, $\overline{X_{h,r}}
 	\cong
 	K_{h-1}\,\dot{\cup}\,
 	r\left(K_h-\left(\frac{h}{2}\right)K_2\right).$ Now, $\Lambda(K_{h-1})
 	=
 	\left\{
 	[h-2]^1,\,
 	[-1]^{h-2}
 	\right\},$ and by \autoref{cocktail_lemma}, $\Lambda\left(K_h-\left(\frac{h}{2}\right)K_2\right)
 	=
 	\left\{
 	[h-2]^1,\,
 	[0]^{h/2},\,
 	[-2]^{h/2-1}
 	\right\}.$ As $\overline{X}_{h,r}$ a disconnected graph with connected components $K_{h-1}$ and $r$ copies of $K_h-\left(\frac{h}{2}\right)K_2$,\\ $$\Lambda\left(\overline{X}_{h,r}\right)
 	=
 	\left\{
 	[h-2]^{r+1},\,
 	[-1]^{h-2},\,
 	[0]^{\frac{rh}{2}},\,
 	[-2]^{r\left(\frac{h}{2}-1\right)}
 	\right\}.$$ Since, $X_{h,r}$ has $(r+1)h-1$ vertices and it is $rh$-regular, its complement is $(h-2)$-regular. For a $k$-regular graph on $n=(r+1)h-1$ vertices, if $k,\lambda_2,\ldots,\lambda_n$ are its eigenvalues, then from \cite{bali2002nska}, the eigenvalues of its complement are $n-k-1,\,-1-\lambda_2,\ldots,-1-\lambda_n$. Applying this relation from \autoref{lem:complement-spectrum} to $\overline{X_{h,r}}$, one copy of the
 	eigenvalue $h-2$ corresponding to the all-one vector gives $(r+1)h-2-(h-2)=rh$ as the principal eigenvalue of $X_{h,r}$. The remaining $r$ copies of $h-2$ give the eigenvalue $-1-(h-2)=1-h.$ Similarly, $-1-(-1)=0, -1-0=-1, -1-(-2)=1.$ Therefore, $\Lambda(X_{h,r})
 	=
 	\left\{
 	[rh]^1,\,
 	[1-h]^r,\,
 	[1]^{r\left(\frac{h}{2}-1\right)},\,
 	[0]^{h-2},\,
 	[-1]^{\frac{rh}{2}}
 	\right\}.$ Hence $X_{h,r}$ is integral.
 	
 	Also, for the case $h=2$, the graph $X_{2,r}\cong K_{2r+1}$. Therefore, the spectrum is $\operatorname{Spec}(X_{2,r})=\{[2r]^1,[-1]^{2r}\}$, which is integral.
 \end{proof}
 
 \begin{theorem}\label{thm:DS-Xhr}
 	The graph $X_{h,r}$ is uniquely determined by its adjacency spectrum. 		
 \end{theorem} 
 
 \begin{proof}
 	Let $Y$ be a graph satisfying $\Lambda(Y)=\Lambda(X_{h,r}).$ We want to prove that $Y\cong X_{h,r}$. 
 	
 	If $h=2$, then $X_{2,r}\cong K_{2r+1}.$ Let $Y$ be a graph such that $\Lambda(Y)=\Lambda(X_{2,r})=
 	\{[2r]^1,[-1]^{2r}\}.$ Hence, $|V(Y)|=2r+1.$ Moreover, $2|E(Y)|=\operatorname{tr}(A(Y)^2)=(2r)^2+2r=2r(2r+1).$
 	Therefore, $|E(Y)|=r(2r+1)=\binom{2r+1}{2}.$ Thus, $Y$ has the maximum possible number of edges on $2r+1$ vertices, and consequently $Y\cong K_{2r+1}\cong X_{2,r}.$
 	
 	For $h\geq 4$, we break the proof into following parts:
 	\begin{enumerate}
 		\item From the spectral information on $Y$, we find the properties of $\overline{Y}$, which is a regular disconnected graph.
 		\item We find the number of vertices and edges in different connected components of $\overline{Y}.$
 		\item From the information about $\overline{Y}$ in (2), we prove that $Y\cong X_{h,r}$.
 	\end{enumerate}
 	
 	\noindent
 	\textbf{Proof of (1)}
 	
 	If $h\geq 4$, we have, by \autoref{Xhr_spectrum},
 	$\Lambda(X_{h,r})
 	=\left\{[rh]^1,\,[1-h]^r,\,[1]^{\,r\left(\frac{h}{2}-1\right)},\,[0]^{h-2},\,[-1]^{\frac{rh}{2}}\right\}= \Lambda(Y).$
 	Therefore, $|V(Y)|=|V(X_{h,r})|=(r+1)h-1.$ and $|E(Y)|=|E(X_{h,r})|=\frac{rh}{2}\{(r+1)h-1\}\,$. Since $X_{h,r}$ is $rh$-regular, the Handshaking Lemma gives $2|E(X_{h,r})| = rh\,|V(X_{h,r})|.$ Consequently, the average degree of $Y$ is $\frac{2|E(Y)|}{|V(Y)|}=rh$.
 	
 	On the other hand, the largest adjacency eigenvalue of $Y$, $\rho(Y)=rh.$ Let $\mathbf{1}$ denote the all-one vector. By \autoref{Rayleigh_quotient}, $\rho(Y) \geq \frac{\mathbf{1}^{T}A(Y)\mathbf{1}}
 	{\mathbf{1}^{T}\mathbf{1}}=\frac{2|E(Y)|}{|V(Y)|}.$
 	Since equality holds, $\mathbf{1}$ is an eigenvector corresponding to $\rho(Y)$. The entries of $A(Y)\mathbf{1}$ are precisely the degrees of the
 	vertices in $Y$. Therefore, every vertex of $Y$ has degree $rh$, and hence $Y$ is $rh$-regular. Thus, the multiplicity of $rh$ is $1$. Hence, the Laplacian matrix $L(Y)$ has only one zero eigenvalue. Following the \autoref{lem:laplacian-connectedness}, $Y$ is connected.
 	
 	Now put $Z=\overline{Y}.$ For every $v\in V(Y)$, $\deg_Z(v) = |V(Y)|-1-\deg_Y(v).$ Hence $\deg_Z(v) = ((r+1)h-1)-1-rh =h-2.$ Therefore $Z$ is $(h-2)$-regular. 
 	
 	Since $Y$ is regular, the usual complement relation for adjacency eigenvalues may be applied. The principal eigenvalue $rh$ of $Y$ gives $|V(Y)|-rh-1=h-2,$ while every other adjacency eigenvalue $\lambda$ of $Y$ is
 	transformed into $-1-\lambda$. Therefore,
 	\[
 	\Lambda(Z)
 	=
 	\left\{
 	[h-2]^{r+1},\,
 	[-1]^{h-2},\,
 	[0]^{\frac{rh}{2}},\,
 	[-2]^{\,r\left(\frac{h}{2}-1\right)}
 	\right\}.
 	\] Since $Z$ is $(h-2)$-regular, the multiplicity of the eigenvalue $h-2$ is the number of connected components of $Z$.
 	
 	\noindent
 	\textbf{Proof of (2)}
 	
 	Let $C$ be an arbitrary connected component of $Z$, and put
 	$d=h-2\text{ and } s=|V(C)|.$ Since $C$ is a connected $d$-regular graph, $d$ is a simple eigenvalue of $C$. All the remaining eigenvalues of $C$ must belong to $\{-1,0,-2\}.$ Suppose, their respective multiplicities are $a,b,c$. Then $\operatorname{Spec}(C) = \{[d]^1,[-1]^a,[0]^b,[-2]^c\}.$ Since the trace of the adjacency matrix of $C$ is zero, $d-a-2c=0,$ such that $a+2c=d.$ Also, since $C$ is $d$-regular on $s$ vertices,
 	$\operatorname{tr}(A(C)^2)=2|E(C)|=sd.$ Using its spectrum, we also have $\operatorname{tr}(A(C)^2) = d^2+a+4c.$ Hence $d^2+a+4c=sd.$ Substituting $a=d-2c$ gives
 	$d^2+d+2c=sd,$ and therefore $2c=d(s-d-1).$ Since a $d$-regular simple graph has at least $d+1$ vertices, put
 	$q=s-d-1\geq0.$ Then $c=\frac{dq}{2}.$ Consequently, $a=d-2c=d-dq=d(1-q).$ Since $a\geq0$ and $d=h-2>0$, we obtain $q\leq1.$ As $q$ is a non-negative integer, $q\in\{0,1\}.$ Thus $s\in\{d+1,d+2\} = \{h-1,h\}.$ Therefore every connected component of $Z$ has either $h-1$ or $h$ vertices. 
 	
 	Let $t$ denote the number of components of $Z$ having $h$ vertices.
 	Since $Z$ has $r+1$ connected components, the remaining $r+1-t$ components have $h-1$ vertices. Hence $|V(Z)|= th+(r+1-t)(h-1).$ Since $|V(Z)|=(r+1)h-1,$ we obtain $th+(r+1-t)(h-1)=(r+1)h-1.$ Simplifying gives $t=r.$
 	Thus $Z$ has exactly one component on $h-1$ vertices and exactly $r$ components on $h$ vertices.
 	The component on $h-1$ vertices is $(h-2)$-regular. Therefore every vertex is adjacent to all the other $h-2$ vertices, and hence this component is isomorphic to $K_{h-1}.$
 	
 	\noindent
 	\textbf{Proof of (3)}
 	
 	Now consider any component of $Z$ having $h$ vertices. It is $(h-2)$-regular. Hence its complement, taken on the same $h$ vertices, is $1$-regular. Since $h$ is even, every $1$-regular graph on $h$ vertices is a collection of $h/2$ disjoint $K_2$ graphs. Therefore, each such component is isomorphic to $K_h-\left(\frac{h}{2}\right)K_2.$ Consequently, $Z \cong K_{h-1} \mathbin{\dot\cup} r\left(K_h-\left(\frac{h}{2}\right)K_2\right).$
 	Taking complements, we obtain
 	\[
 	Y
 	\cong
 	\overline{K}_{h-1}
 	\vee
 	\left(
 	\bigvee_{i=1}^{r}
 	\left(\frac{h}{2}\right)K_2
 	\right)
 	=
 	X_{h,r}.
 	\]
 	Hence $X_{h,r}$ is determined by its adjacency spectrum.
 \end{proof}
 
 \section{Group Realizations of the Graph $X_{h,r}$}
 
 The purpose of this section is to solve the realization problem for the graph family $X_{h,r}$. More precisely, we seek to characterize those finite $2$-groups $G$ for which $D^*(G)\cong X_{h,r}.$ After obtaining an exact realization criterion, we investigate the extent to which the realizing group is uniquely determined by the graph. In particular, we show that distinct Abelian and non-Abelian groups may give rise to the same graph $X_{h,r}$.

\begin{theorem}\label{forward_realization}
		Let $\exp(G)=4$, and $H=\{x\in G:x^2=e\}$ is a subgroup of $G$. Put $h=|H|, r=[G:H]-1.$ Then $D^{*}(G)\cong X_{h,r}.$
\end{theorem}
	  
 	\begin{proof}
 		Since $H=\{x\in G:x^2=e\}$
 		is assumed to be a subgroup, every nonidentity element of $H$
 		has order $2$. Moreover, $H$ is abelian. Indeed, for any
 		$x,y\in H$, we have $xy=(xy)^{-1}=y^{-1}x^{-1}=yx.$
 		Therefore, $H$ is an elementary abelian $2$-group.
 		
 		 Let $\varphi\in\operatorname{Aut}(G)$ and $x\in H$. Then $\varphi(x)^2 =\varphi(x^2) = \varphi(e)=e.$
 		Hence $\varphi(x)\in H$. Thus, $H$ is invariant under every
 		automorphism of $G$, and therefore $H$ is characteristic in
 		$G$. Consequently, $H$ is a normal subgroup of $G$, that is $H\trianglelefteq G$.
 
 		Since $G$ has exponent $4$, we have $x^4=e$ for every $x\in G$. Therefore, $(x^2)^2=e,$ and hence $x^2\in H.$ It follows that $(xH)^2=x^2H=H$ for every $x\in G$. Thus, every nonidentity element of $G/H$ has order $2$. Consequently, $G/H$ is an elementary Abelian
 		$2$-group.
 		
 		Put $q=[G:H]
 		\text{ and } r=q-1.$ Then $G/H=\{H,C_1,C_2,\ldots,C_r\},$
 		where $C_1,\ldots,C_r$ are the nontrivial cosets of $H$.
 		Moreover, $|C_i|=|H|=h$ for every $i\in\{1,\ldots,r\}$. Let $C_i=xH$ be a non-trivial coset of $H$. Since $G/H$ is a non-trivial elementary Abelian $2$-group, every non-identity element of $G/H$ has order $2$. Hence, $(C_i)^2=(xH)^2=x^2H=H$, which implies $C_iC_i=H$. 
 			
 			Also, for $C_i=xH$ and $C_j=yH$. Then, $C_iC_j=(xH)\cdot (yH)=(xy)H$. If we assume that $C_iC_j=H$, then $xy\in H$. Equivalently, $yH=x^{-1}H$. Therefore, $C^{-1}_i=C_j$. Since every non-identity element of $G/H$ is of order $2$, thus, $C_i=C^{-1}_i$. Hence, $C_i=C_j$ contradicting $i\neq j$. Thus, $C_iC_j\neq H$ for $i\neq j$. Therefore, $C_iC_j=H$ is impossible for two distinct non-trivial cosets. 
 		
 		Since $G$ is a finite $2$-group, by
 		\autoref{complement_POE}, two distinct vertices $x,y\in G^*$ are adjacent in $D^*(G)$ if and only if $o(xy)\neq2.$

 		We now establish the required structural properties of the graph $D^*(G)$.
 		
 		\medskip
 		\noindent
 		\textbf{(1) The set $H^*$ is independent in the graph $D^*(G)$.}
 		
 		Let $x,y\in H^*$ with $x\neq y$. Since, $H$ is a subgroup, $xy\in H.$ Also, $xy\neq e$, because $xy=e$ would imply $y=x^{-1}=x,$
 		contrary to $x\neq y$. Hence $xy\in H^*$ and $o(xy)=2.$
 		Therefore, $x$ and $y$ are not adjacent in $D^*(G)$, and
 		$H^*$ is independent set of vertices in the graph.
 		
 		\medskip
 		\noindent
 		\textbf{(2) The set of vertices $H^*$ is completely joined to every nontrivial coset $C_i$.}
 		
 		Let $x\in H^*$ and $y\in C_i$. Since $x\in H$ and $H$ is
 		normal, $xy\in C_i.$ Because $C_i$ is a nontrivial coset, it is disjoint from $H$.Thus, $xy\notin H.$
 		Every element outside $H$ has order $4$, and hence $o(xy)=4.$ Therefore, $x\sim_{D^*(G)}y$. Thus,
 		$H^*$ is completely joined to every $C_i$.
 		
 		\medskip
 		\noindent
 		\textbf{(3) Each $C_i$ induces a subgraph isomorphic to $\frac{h}{2}K_2$ in $D^*(G)$.}
 		
 		Fix $i\in\{1,\ldots,r\}$ and let $x,y\in C_i$ be distinct.
 		Since $C_iC_i=H$, $xy\in H.$
 		Hence $o(xy)$ is either $1$ or $2$. Thus, $x\sim_{D^*(G)}y
 		\quad\Longleftrightarrow\quad
 		o(xy)=1
 		\quad\Longleftrightarrow\quad
 		xy=e
 		\quad\Longleftrightarrow\quad
 		y=x^{-1}.$
 		
 		Since $x\notin H$, the element $x$ has order $4$, and therefore $x\neq x^{-1}.$
 		Moreover, $	x^{-1}H=xH,$
 		so $x^{-1}\in C_i$. Hence the inverse map partitions $C_i$
 		into $\frac h2$ disjoint inverse pairs. Therefore, $D^*(G)[C_i]\cong \left(\frac{h}{2}\right)K_2.$
 		
 		\medskip
 		\noindent
 		\textbf{(4) Every Element of $C_i$ is adjacent to every element of $C_j$ for $i\neq j$.}

 		Since $C_iC_j\neq H$, for $i\neq j$, $C_iC_j$ is another nontrivial coset, say $C_k$.
 		Hence $xy\in C_k\subseteq G\setminus H.$ Thus, $o(xy)=4,$
 		and consequently $x$ and $y$ are adjacent in $D^*(G)$.
 		Therefore, every pair of distinct nontrivial cosets is
 		completely joined. It follows that $V(D^*(G)) = H^*\cup C_1\cup\cdots\cup C_r,$ where
 		$D^*(G)[H^*]\cong\overline K_{h-1}$
 		and $D^*(G)[C_i]\cong {\left(\frac{h}{2}\right)K_2}$.  Since every pair of distinct parts is completely joined, we obtain $D^*(G) \cong \overline K_{h-1} \vee \left(\bigvee_{i=1}^{r}{\left(\frac{h}{2}\right)K_2}\right) =X_{h,r}$.

 	\end{proof}

 \begin{corollary}\label{forward_integral}
 	Under the hypotheses of \autoref{forward_realization},
 	the graph $D^{*}(G)$ is connected, $rh$-regular, and integral with  $$\Lambda(D^{*}(G))
 	=
 	\left\{
 	[rh]^1,\,
 	[1-h]^r,\,
 	[1]^{r\left(\frac{h}{2}-1\right)},\,
 	[0]^{h-2},\,
 	[-1]^{\frac{rh}{2}}
 	\right\}.$$
 	\end{corollary}
 \begin{proof}
 	By \autoref{forward_realization}, $D^{*}(G)\cong X_{h,r}.$ The result follows from \autoref{X_h_r_basic} and \autoref{Xhr_spectrum}.
 	\end{proof}
 
\begin{theorem}\label{converse_realization}
 	 If $D^{*}(G)\cong X_{h,r},$ then $\exp(G)=4$, and $H=\{x\in G:x^{2}=e\}$ is a subgroup of $G$. Moreover, $|H|=h
 	\text{ and }
 	[G:H]=r+1.$
 	\end{theorem}
 \begin{proof}
 	Suppose $D^{*}(G)\cong X_{h,r}.$ By \autoref{X_h_r_basic}, the graph $X_{h,r}$ has $(r+1)h-1$ vertices and is $rh$-regular. Hence $|G|=(r+1)h.$ Since $D^*(G)\cong X_{h,r}$ is regular, \autoref{exponent_corollary} implies that $\exp(G)\leq 4.$
 	We first discard the case $\exp(G)=2$. If $\exp(G)=2$, every non-identity element of $G$ is an involution, and hence $r_2(G)=|G|-1.$ By \autoref{degree_theorem}, every vertex of $D^*(G)$ then has degree $|G|-1-r_2(G)=0,$ which contradicts the fact that $X_{h,r}$ is $rh$-regular with $r,h\geq1$. Therefore, $\exp(G)=4.$
 	
 	Consequently, every non-identity element of $G$ has order either $2$ or $4$. Hence, by \autoref{degree_theorem}, every vertex of $D^*(G)$ has degree $|G|-1-r_2(G).$ Since $D^*(G)$ is $rh$-regular, we obtain $rh=|G|-1-r_2(G).$ Using $|G|=(r+1)h$, it follows that $rh=(r+1)h-1-r_2(G),$ and therefore $r_2(G)=h-1.$ Thus $|H|=r_2(G)+1=h.$ 
 	
 	Now, we shall prove that $H$ is a subgroup of $G$. If $h=2$, then $G$ has exactly one involution. Hence $H$ consists of
 	the identity and this unique involution. Therefore, $H$ is a
 	subgroup of order $2$. Putting $h\geq 4$ in \autoref{complement_POE}, we have, $\Gamma^{*}(G)\cong \overline{X_{h,r}}.$ From \autoref{X_h_r_basic}, $\overline{X_{h,r}}
 	\cong
 	K_{h-1}\,\dot{\cup}\,
 	r\left(K_h-{\left(\frac{h}{2}\right)K_2}\right).$ For $h\geq4$, each graph $K_h-{\left(\frac{h}{2}\right)K_2}$ is connected. Choose a central involution $z\in Z(G)$. Such an element exists since
 	$G$ is a nontrivial finite $2$-group. If $t$ is any involution
 	distinct from $z$, then $zt=tz$ and $(zt)^2=e$. Since $zt\neq e$, we have $o(zt)=2$. Hence $z\sim_{\Gamma^{*}(G)}t.$ Thus, all $h-1$ involutions of $G$ belong to the same connected
 	component of $\Gamma^{*}(G)$.
 	
 	We claim that every two involutions of $G$ commute. Suppose, to the
 	contrary, that $a$ and $b$ are non-commuting involutions. Then
 	$o(ab)=4$. Put $c=ab
 	\text{ and }
 	u=c^2.$ The element $u$ is an involution and is central in
 	$\langle a,b\rangle$. Hence $b'=bu$ is also an involution, with $b'\neq b$. Moreover, $ab'=abu=cu=c^3=c^{-1},$whereas $b'a=bua=bau=c^{-1}u=c.$ Since $c\neq c^{-1}$, the involutions $a$ and $b'$ do not commute.
 	Thus $a$ fails to commute with at least two distinct involutions, namely $b$ and $b'$.
 	
 	On the other hand, all involutions lie in one connected component of $\Gamma^{*}(G)$. If this component is $K_{h-1}$, then all involutions are pairwise adjacent and hence pairwise commute. If it is one of the graphs $K_h-{\left(\frac{h}{2}\right)K_2}$, then every vertex has exactly one non-neighbour
 	within that component. Consequently, an involution can fail to
 	commute with at most one other involution. This contradicts the
 	preceding claim. Therefore all involutions of $G$ commute pairwise.
 	
 	Now, let $x,y\in H$. If either $x=e$ or $y=e$, then $xy\in H$. If $x$ and $y$ are non-identity involutions, then $xy=yx$, and hence $(xy)^2=x^2y^2=e.$ Thus $xy\in H$. Therefore $H$ is closed under multiplication.
 	Since every element of $H$ is its own inverse, $H$ is a subgroup of $G$. Finally, $[G:H] =\frac{|G|}{|H|} =\frac{(r+1)h}{h}
 	= r+1.$ Hence $|H|=h
 	\text{ and }
 	[G:H]=r+1.$
 	\end{proof}

 	\begin{theorem}\label{exact_characterization}
 	 $D^{*}(G)\cong X_{h,r}$ if and only if $\exp(G)=4$ and $H(G)=\{x\in G:x^{2}=e\}$ is a subgroup of $G$ satisfying $|H(G)|=h
 	\text{ and }
 	[G:H(G)]=r+1.$ Equivalently, $D^{*}(G)
 	\cong
 	\overline{K}_{h-1}\vee
 	\left(
 	\bigvee_{i=1}^{r}{\left(\frac{h}{2}\right)K_2}
 	\right)$ if and only if $G$ has exponent $4$ and the identity together with all involutions of $G$ forms a subgroup of order $h$ and index $r+1$. Also, the adjacency spectrum of $D^*(G)$ determines $|G|,r_2(G),|H(G)|$, and $[G:H(G)]$; precisely, if $n=|V(D^*(G))|$ and $\rho$ is its adjacency spectral radius, then $|G|=n+1; |H(G)|=n-\rho+1, r_2(G)=n-\rho,$ and $[G:H(G)]=1+\frac{\rho}{n-\rho+1}$.
 	\end{theorem}

 	\begin{proof}
 	Let $\exp(G)=4$ and $H(G)=\{x\in G:x^2=e\}$ is a subgroup satisfying $|H(G)|=h,[G:H(G)]=r+1.$ Then \autoref{forward_realization} gives $D^{*}(G)\cong X_{h,r}.$

 	Conversely, let $D^{*}(G)\cong X_{h,r},$ then by \autoref{converse_realization}, $\exp(G)=4, H(G)\leq G$, together with $|H(G)|=h \text{ and } [G:H(G)]=r+1$. Hence the two conditions are equivalent. Also, $r_2(G)=h-1$.
 	
 	By \autoref{X_h_r_basic}, the graph $X_{h,r}$ is
 		$rh$-regular and has $n=(r+1)h-1$ vertices. Since $X_{h,r}$ is connected, its spectral radius equals
 		its regular degree. Hence $\rho=rh.$ Therefore, $n-\rho
 		=(r+1)h-1-rh = h-1,$ which gives $h=n-\rho+1.$ Since $\rho=rh$, it follows that $r=\frac{\rho}{h}=\frac{\rho}{\,n-\rho+1\,}$ and so, $[G:H(G)]=r+1=1+ \frac{\rho}{\,n-\rho+1\,}$. Thus both parameters are determined by the adjacency spectrum. 
 	\end{proof}
 	
 	\begin{remark}\label{spectral_information_loss}Although the adjacency spectrum determines the parameters $h$ and $r$, and therefore the quantities above, it does not in general determine the isomorphism type of the realizing group. Indeed, there are several non-isomorphic groups having exactly the same graph and hence the same adjacency spectrum.
 	\end{remark}
 	
\begin{corollary}\label{parameter_restriction}
 Let $D^{*}(G)\cong X_{h,r}$.Then, there exist positive integers $m$ and $d$ such that $h=2^{m} \text{ and } r=2^{d}-1.$ Equivalently, every graph in the family $X_{h,r}$ that is realizable as $D^{*}(G)$ has the form $X_{2^{m},\,2^{d}-1}.$
 \end{corollary}

\begin{proof}
	By \autoref{exact_characterization}, $H=H(G)=\{x\in G:x^{2}=e\}$ is a subgroup of $G$ satisfying $|H|=h \text{ and } [G:H]=r+1$. Every element of $H$ has order at most $2$. Hence $H$ is an
	elementary abelian $2$-group. Therefore, $H\cong \mathbb{Z}_{2}^{\,m}$ for some positive integer $m$, and consequently $h=|H|=2^{m}.$ Moreover, $H$ is characteristic in $G$, because every automorphism of
	$G$ preserves the property $x^{2}=e$. Hence $H\trianglelefteq G$ and
	the quotient group $G/H$ is well defined.
	Since $\exp(G)=4$, for every $x\in G$ we have $(x^{2})^{2}=x^{4}=e,$ and thus, $x^2\in H$. It follows that $(xH)^{2}=x^{2}H=H$ for every $xH\in G/H$. Thus, every non-identity element of $G/H$ has
	order $2$, which implies that, $G/H$ is an elementary Abelian $2$-group.
	Therefore, $G/H\cong \mathbb{Z}_{2}^{\,d}$ for some positive integer $d$. Consequently, $r+1=[G:H]=|G/H|=2^{d},$ and hence $r=2^d-1$.  
	\end{proof}

 	\begin{remark}\label{realizable_family}
 		The family $X_{h,r}$ is defined purely graph-theoretically for every even integer $h\geq 2$ and every $r\geq 1$. \autoref{Xhr_spectrum} shows that each such graph is integral. However, \autoref{parameter_restriction} shows that only a restricted subfamily can arise as a Proper POE Complement Graph of a finite $2$-group. In particular, a realizable graph must have the form $X_{2^{m},\,2^{d}-1}.$ This leads us to the following  realization problem:   
 		\begin{enumerate}
 			\item for which pairs
 			$(m,d)$, does there exist a finite $2$-group $G$ such that $D^{*}(G)\cong X_{2^{m},\,2^{d}-1}$?
 			\item how many non-isomorphic groups can realize the same graph?
 		\end{enumerate}   
 		\end{remark}
 	
 	\subsection{Realizations of the Graphs $X_{2^m,\,2^d-1}$}
 	
 	By \autoref{parameter_restriction}, every member of the family $X_{h,r}$ that arises as a \emph{Proper POE Complement Graph} of a finite $2$-group must have the form $X_{2^m,\,2^d-1}$, where $m,d\in \mathbb{N}$. We now check the converse realization problem. We first determine the Abelian realization completely and then construct non-Abelian realizations. The graph $X_{2^m,\,2^d-1}$ has $|V(X_{2^m,\,2^d-1})|=2^{m+d}-1$ vertices and $|E(X_{2^m,\,2^d-1})| = 2^{m-1}(2^d-1)(2^{m+d}-1)$ edges. One specific example of non-Abelian realization is shown in \autoref{order16-full}.

 	\subsubsection{Abelian Realizations}
 	
 \begin{theorem}\label{abelian_realization}
 			The graph $X_{2^m,\,2^d-1}$ admits an Abelian realization as $D^*(G)$ for a finite Abelian $2$-group $G\cong \mathbb Z_4^{\,d}\times \mathbb Z_2^{\,m-d}$ which is unique up to isomorphism if and only if $1 \leq d\leq m$. 
 	\end{theorem}
 	
 	\begin{proof}
 	
 			Let $1\leq d\leq m$, and set $G=\mathbb Z_4^{\,d}\times
 			\mathbb Z_2^{\,m-d}.$ Clearly, $\exp(G)=4$. Let $H(G)=\{x\in G:x^2=e\}.$ In each copy of $\mathbb Z_4$, exactly two elements are annihilated by
 			squaring, while every element of $\mathbb Z_2$ has order at most $2$.
 			Hence $H(G)\cong
 			\mathbb Z_2^{\,d}\times
 			\mathbb Z_2^{\,m-d}
 			\cong
 			\mathbb Z_2^{\,m}.$ Therefore, $|H(G)|=2^m.$ Moreover, $|G|
 			= 4^d\,2^{m-d}
 			= 2^{m+d},$ and consequently, $[G:H(G)]
 			= \frac{2^{m+d}}{2^m}
 			= 2^d.$ By \autoref{exact_characterization}, $D^{*}(G)
 			\cong
 			X_{2^m,\,2^d-1}.$
 		
 		Conversely, let $G$ be a finite Abelian $2$-group satisfying $D^{*}(G)\cong X_{2^m,\,2^d-1}.$ By \autoref{exact_characterization}, $\exp(G)=4,
 			|H(G)|=2^m,
 			[G:H(G)]=2^d.$ Since, $G$ is a finite Abelian $2$-group of exponent $4$, there exist
 			non-negative integers $a,b$, with $a\geq1$, such that $G\cong
 			\mathbb Z_4^{\,a}\times
 			\mathbb Z_2^{\,b}.$ For this group, $H(G)\cong
 			\mathbb Z_2^{\,a+b},$ and hence $|H(G)|=2^{a+b}.$ Also, $[G:H(G)]=2^a.$ Comparing these expressions with $|H(G)|=2^m
 			\text{ and }
 			[G:H(G)]=2^d,$ we obtain $a=d$ and $a+b=m$. Therefore, $b=m-d$. Since $b\geq0$, necessarily $d\leq m$. Hence $G\cong
 			\mathbb Z_4^{\,d}\times
 			\mathbb Z_2^{\,m-d}.$ Combining we get, an Abelian realization exists, and is unique up to isomorphism, if and only if $1\leq d\leq m$.
 			\end{proof}
 	
 	\subsubsection{Non-Abelian Realizations}
 	
 	\begin{theorem}\label{Q_8_realization}
 			If $ G_m=Q_8\times \mathbb Z_2^{\,m-1},$ then $D^*(G_m)\cong X_{2^m,\,3}$, as graphs, when $m\geq 1$.
 	\end{theorem}
 	
 	\begin{proof}
 		Since $Q_8$ is non-Abelian, $G_m$ is non-Abelian and $\exp(G_m)=4$. The only elements of $Q_8$ whose squares are the identity are $1$ and $-1$. Since every element of $\mathbb Z_2^{\,m-1}$ has order at most $2$, we obtain $ H(G_m) = \{x\in G_m:x^2=e\} = \{1,-1\}\times\mathbb Z_2^{\,m-1}.$ Hence $H(G_m)\cong \mathbb{Z}_{2}^m$ and therefore $|H(G_m)|=2^m$. Furthermore, $|G_m|= |Q_8|\,|\mathbb Z_2^{\,m-1}|=2^3\cdot 2^{m-1}=2^{m+2}$. Thus, $[G_m:H(G_m)] = \frac{2^{m+2}}{2^m} = 4.$ Consequently, $[G_m:H(G_m)]-1=3$. By \autoref{exact_characterization}, $D^*(G_m)\cong X_{2^m,\,3}$.
 	\end{proof}
 	
 	The above theorem indicates that for every integer $m\geq 1$, the graph $X_{2^m,3}$ admits a non-Abelian realization. Moreover, for every $m\geq 2$, \autoref{abelian_realization} and \autoref{Q_8_realization} indicate that the graph $X_{2^m,3}$ admits both Abelian and Non-Abelian realization. More precisely, $D^*\left(\mathbb Z_4^2\times\mathbb Z_2^{m-2}\right) \cong D^*\left(Q_8\times\mathbb Z_2^{m-1}\right) \cong X_{2^m,3}.$ The two realizing groups are non-isomorphic, since the first is Abelian whereas the second is non-Abelian.

 	\begin{theorem}\label{general_nonabelian_realization}
 			 Consider $G = Q_8^{\,a}\times
 			\mathbb Z_4^{\,c}\times \mathbb Z_2^{\,b}$, where $a\geq 1$ and $b,c\geq 0$.  Set $m=a+b+c \text{ and } d=2a+c.$ Then $G$ is non-Abelian and $D^*(G)\cong X_{2^m,\,2^d-1}$. 
 	\end{theorem}

 \begin{proof}
 	    Since $a\geq1$, the group $G$ contains a direct factor isomorphic to $Q_8$, and hence $G$ is non-Abelian. Moreover, $\exp(G)=4$. For each copy of $Q_8$, the elements whose squares are the identity, form the subgroup $\{1,-1\}\cong \mathbb{Z}_2$. In each copy of $\mathbb{Z}_4$ the elements whose squares are the
 		identity form a subgroup isomorphic to $\mathbb Z_2$, while every element of $\mathbb Z_2$ has square equal to the identity. Therefore, $H(G)=\{x\in G:x^2=e\}
 		\cong \mathbb Z_2^{\,a+c+b}$ and $|H(G)|
 		= 2^{a+b+c} =2^m.$ On the other hand, $|G|
 		= 8^a4^c2^b =2^{3a+2c+b}.$ Consequently, $[G:H(G)]
 		=2^{3a+2c+b-(a+c+b)} =2^{2a+c}=2^d.$ Thus, $[G:H(G)]-1=2^d-1.$ By \autoref{exact_characterization}, $D^{*}(G) \cong
 		X_{2^m,\,2^d-1}.$
 \end{proof}
 
 \begin{corollary}\label{non-abelian_range}
 		Let $m,d\in \mathbb{N}$ such that $2\leq d\leq 2m.$ Then the graph $X_{2^m,\,2^d-1}$ admits a non-Abelian realization.
 \end{corollary}
 
 \begin{proof}
 	We divide the proof into two cases.
 	
 	\textbf{Case-I:} Let $2\leq d \leq m+1$. Set $a=1, c=d-2, b=m-d+1$. Then $a\geq 1$ and since $2\leq d \leq m+1$, we have, $b,c\geq 0$. Moreover, $a+b+c=1+(m-d+1)+(d-2)=m$ and $2a+c=2+d-2=d$. Thus, by \autoref{general_nonabelian_realization}, $D^{*}
 		\left(
 		Q_8\times
 		\mathbb Z_4^{\,d-2}\times
 		\mathbb Z_2^{\,m-d+1}
 		\right)
 		\cong
 		X_{2^m,\,2^d-1}.$
 	
 	\textbf{Case-II:} Let $m+1\leq d\leq 2m$. Now, we choose $a=d-m$ and $c=2m-d$, $b=0$. Then, $a\geq 1$ and $c\geq 0$, and $a+c=d-m+2m-d=m$, while $2a+c=2(d-m)+(2m-d)=d$. Again, \autoref{general_nonabelian_realization} tells that $D^{*}
 		\left(
 		Q_8^{\,d-m}\times
 		\mathbb Z_4^{\,2m-d}
 		\right)
 		\cong
 		X_{2^m,\,2^d-1}.$ Therefore, every pair satisfying $2\leq d \leq 2m$ has a non-Abelian realization. 
 \end{proof}
 
 Combining \autoref{abelian_realization} and \autoref{non-abelian_range}, we obtain that, for
 $2\leq d\leq m$, the graph $X_{2^m,\,2^d-1}$ admits both an
 Abelian and a non-Abelian realization. More precisely, $D^*\left( \mathbb Z_4^{\,d}\times \mathbb Z_2^{\,m-d} \right)
 \cong D^*\left( Q_8\times \mathbb Z_4^{\,d-2}\times \mathbb Z_2^{\,m-d+1} \right) \cong X_{2^m,\,2^d-1}.$ Thus, for $2\leq d\leq m$, the same graph $X_{2^m,\,2^d-1}$ is realized by at least two non-isomorphic finite $2$-groups, one Abelian and one non-Abelian.

 	\subsubsection{Index-Two Non-Abelian Realizations}
 	
 	\begin{theorem}\label{index_two_nonabelian}
 		Let $\theta\in\operatorname{Aut}(\mathbb Z_2^{\,m-1})$ for $m\geq 3$ be a nontrivial automorphism of
 		order $2$, and let $G_{m,\theta}={\mathbb Z_2^{\,m-1}}\rtimes_{\theta}\mathbb Z_4,$
 		where a generator $g$ of $\mathbb Z_4$ acts on $\mathbb Z_2^{\,m-1}$ as $\theta$.
 		Then $G_{m,\theta}$ is non-Abelian and $D^{*}(G_{m,\theta})\cong X_{2^m,\,1}.$
 	\end{theorem}
 
 		\begin{proof}
 			Since $\theta$ is nontrivial, the action of $\mathbb Z_4$ on $Z_2^{\,m-1}$ is
 			nontrivial, and hence $G_{m,\theta}$ is non-Abelian.
 			
 			Since $\theta^2=\operatorname{id}_{Z_2^{\,m-1}}$, where $\operatorname{id}_{Z_2^{\,m-1}}$ is the identity automorphism on $Z_2^{\,m-1}$, the element $g^2$ acts
 			trivially on $Z_2^{\,m-1}$. We claim that $H(G_{m,\theta})
 			=
 			\{x\in G_{m,\theta}:x^2=e\}
 			=
 			Z_2^{\,m-1}\times\langle g^2\rangle.$
 			
 			Indeed, every element of $\mathbb{Z}_2^{m-1}$ has order at most $2$. Moreover, for
 			$v\in Z_2^{\,m-1}$, $(vg^2)^2
 			=
 			v\theta^2(v)g^4
 			=
 			v^2
 			=
 			e.$
 			Thus $Z_2^{\,m-1}\times\langle g^2\rangle
 			\subseteq H(G_{m,\theta}).$
 			
 			On the other hand, $(vg)^2=v\theta(v)g^2\neq e,$
 			because its $\mathbb Z_4$-component is $g^2\neq e$. Similarly, $(vg^3)^2=v\theta(v)g^2\neq e.$
 			Hence no element in either of the cosets $\mathbb{Z}_2^{m-1}g$ or $Z_2^{\,m-1}g^3$ belongs to
 			$H(G_{m,\theta})$. Therefore, $H(G_{m,\theta})
 			=
 			Z_2^{\,m-1}\times\langle g^2\rangle.$
 			Consequently, $H(G_{m,\theta})
 			\cong
 			\mathbb Z_2^{\,m},$ so $|H(G_{m,\theta})|=2^m.$
 			Also, $|G_{m,\theta}|=4|Z_2^{\,m-1}|=2^{m+1},$
 			and hence $[G_{m,\theta}:H(G_{m,\theta})]=2.$
 			Since $g^2$ acts trivially on $\mathbb{Z}_2^{m-1}$, it commutes with every element of $\mathbb{Z}_2^{m-1}$. Moreover, $\mathbb{Z}_2^{m-1}$ is elementary Abelian, so $(v\theta(v))^2=e$ for every $v\in \mathbb{Z}_2^{m-1}$. Hence $(vg)^4
 			=\left(v\theta(v)g^2\right)^2=e,(vg^3)^4
 			=\left(v\theta(v)g^2\right)^2=e.$
 			Thus every element of $G_{m,\theta}$ has order dividing $4$. Since $g$ has order $4$, we conclude that
 			$\exp(G_{m,\theta})=4.$ Therefore, by \autoref{exact_characterization}, $D^{*}(G_{m,\theta})
 			\cong X_{2^m,\,2^1-1}=X_{2^m,\,1}.$
 	\end{proof}

 	\begin{remark}\label{index_two_collision}
 		By \autoref{abelian_realization} and \autoref{index_two_nonabelian} we can say that for every integer $m\geq3$, the graph $X_{2^m,\,1}$
 		admits both an Abelian and a non-Abelian realization. In particular, $D^{*}\left(
 		\mathbb Z_4\times\mathbb Z_2^{\,m-1}
 		\right)
 		\cong
 		D^{*}(G_{m,\theta})
 		\cong
 		X_{2^m,\,1}.$
 	\end{remark}

 	\begin{remark}\label{full_collision_range}
 		For every integer $d$ satisfying $1\leq d\leq m,$ with $m\geq 3$, the graph $X_{2^m,\,2^d-1}$ admits both an Abelian and a non-Abelian realization.
 	\end{remark}
 	
 	\begin{theorem}\label{realization_summary}
 		Let $m$ and $d$ be positive integers.
 		
 		\begin{enumerate}
 			\item If $1\leq d\leq 2m,$
 			then the graph $X_{2^m,\,2^d-1}$ is realizable as $D^{*}(G)$ for some finite $2$-group $G$.
 			
 			\item The graph $X_{2^m,\,2^d-1}$ admits an Abelian realization if and only if $1\leq d\leq m.$ In this case the Abelian realization is unique up to isomorphism and is $\mathbb Z_4^{\,d}\times
 			\mathbb Z_2^{\,m-d}.$
 			\item If $2\leq d\leq2m,$ then $X_{2^m,\,2^d-1}$ admits a non-Abelian realization.
 			\item If $m<d\leq2m,$ then $X_{2^m,\,2^d-1}$ admits a non-Abelian realization but admits no Abelian realization.
 		\end{enumerate}
 	\end{theorem}
 	
 		\begin{proof}
 			(1) follows by combining
 			\autoref{abelian_realization} with
 			\autoref{non-abelian_range}.
 			
 			(2) is precisely \autoref{abelian_realization}.
 			
 			(3) follows from \autoref{non-abelian_range}.

 			Finally, if $m<d\leq2m$, then
 			\autoref{non-abelian_range} provides a non-abelian
 			realization, whereas \autoref{abelian_realization} shows that no abelian realization is possible. This proves (4).
 		\end{proof}

 		\begin{remark}[An inverse problem for algebraically defined graphs]
 			\label{rem:inverse-problem}
 			A natural inverse problem in algebraic and spectral graph theory is the following: \emph{Can two non-isomorphic algebraic objects give rise to isomorphic graphs under the same graph construction?}
 			
 			For the graph construction considered in this article, the answer is affirmative. Indeed, for $2\leq d\leq m$, consider $G_1=\mathbb Z_4^d\times\mathbb Z_2^{\,m-d}$ and $G_2= Q_8\times \mathbb Z_4^{\,d-2}\times \mathbb Z_2^{\,m-d+1}.$ The group $G_1$ is Abelian whereas $G_2$ is non-Abelian, and hence $G_1\not\cong G_2.$ Nevertheless, $D^*(G_1) \cong D^*(G_2)\cong X_{2^m,\,2^d-1}.$
 			
 			Similarly, for every $m\geq3$, the Abelian group
 				$\mathbb Z_4\times\mathbb Z_2^{\,m-1}$ and the non-Abelian group $\mathbb Z_2^{\,m-1}\rtimes_{\theta}\mathbb Z_4$ give rise to the same graph $X_{2^m,1}$. Thus, the correspondence $G\longmapsto D^*(G)$ is not injective on isomorphism classes of finite groups. In particular, the associated graph, and consequently its adjacency spectrum, does not in general determine the isomorphism type of the underlying group. Hence the family provides infinitely many explicit examples in which distinct algebraic structures are indistinguishable at the level of the associated graph and its adjacency spectrum. This raises the broader inverse problem of determining which group-theoretic information is preserved by $D^*(G)$, and of classifying all non-isomorphic groups that give rise to a fixed graph $X_{h,r}$.
 		\end{remark}

 \section{Scalable Family of Integral Graphs and Spectral Determination}
 
 In this section, we investigate how new members of the family arise from group-theoretic constructions. We subsequently translate these results to the Proper POE Complement Graphs of their realizing groups.
 
	\subsection{Construction of Scalable Family of Integral Graphs}
	
	We now describe a graph-theoretic procedure for obtaining the larger integral graphs from smaller integral graphs.
	
	\begin{procedure}[Recursive construction in the $h$-direction]\label{construction_procedure}
		We have, $V(X_{h,r})= V_0^{(h)}\dot\cup V_1^{(h)}\dot\cup\cdots\dot\cup V_r^{(h)},$ where $X_{h,r}[V_0]\cong\overline K_{h-1}$
		and $X_{h,r}[V_i]\cong {\left(\frac{h}{2}\right)}K_2, \text{ for } i=1,2,\cdots, r.$ Follow the below steps to construct a new graph $Y$.
		
		\begin{enumerate}
			\item \textbf{Vertex Addition:} 
			\begin{enumerate}
				\item Define $V_0'=V_0^{(h)}\cup B_0$, where $B_0$ is a set of $h$ vertices.
				
				\item For each $i=1,\ldots,r$, define $V_i'=V_i^{(h)}\cup B_i$, where each $B_i$ is a set of $h$ new vertices.  
			\end{enumerate}

			\item \textbf{Edge Addition:}
			\begin{enumerate}
				\item For each $i =1,\ldots,r $, pair the $h$ vertices of $B_i$ by $\frac{h}{2}$ disjoint edges, such that
				$Y[B_i]\cong \left(\frac h2\right)K_2.$ 
				\item For every $i,j$ with $0\leq i<j\leq r$, join every vertex of $V_i'$ to every vertex of $V_j'$.
			\end{enumerate}
		\end{enumerate}	
	\end{procedure}
	
	\begin{theorem}
		The graph generated by \autoref{construction_procedure} is isomorphic to $X_{2h,r}$.
	\end{theorem}

		\begin{proof}
			
		Initially, the graph $X_{2,r}$ has $\text{ vertex set }  V(X_{2,r}) =V_0\dot\cup V_1\dot\cup\cdots\dot\cup V_r$, where, $V_0=\{u\}$ and $V_i=\{x_i,y_i\}, 1\le i\le r$ and $\text{ edge set } E(X_{2,r})=\left(\bigcup_{i=1}^{r}\bigl\{(x_i,y_i)\bigr\}
			\right)\cup\left(\bigcup_{i=1}^{r}\bigl\{(u,x_i),(u,y_i)\bigr\}\right)\cup\left(\bigcup_{1\le i<j\le r}\bigl\{(a,b):a\in V_i,\ b\in V_j\bigr\}\right).$
		 Recall that, every subset $V_1^{(h)}, \cdots, V_r^{(h)}$ of $V(X_{h,r})$ contains $h$ and $V_0^{(h)}$ contains $h-1$ vertices. Also, $V(X_{2h,r})=V_0^{(2h)}\cup V_1^{(2h)}\cup \cdots \cup V_r^{(2h)}$ where every subset $V_i^{(2h)}, i=1,2,\ldots, r$ contains $2h$ and $V_0^{(2h)}$ contains $2h-1$ vertices. As $V_i'=V_i^{(h)}\cup B_i$ for $i=0,1,2,\cdots, r$ and each $B_i$ contains $h$ vertices, the number of vertices in $V_i'$ and in $V_i^{(2h)}$ are equal.

		The induced subgraph $X_{h,r}[V_0^{(h)}]$ has no edge. \autoref{construction_procedure} does not add any edge between two vertices of $V_0'$. Therefore, the induced subgraph $Y[V_0']$ and $X_{2h,r}[V_0^{(2h)}]$ are isomorphic. The induced subgraph $X_{h,r}[V_i^{(h)}]$ is isomorphic to $\frac{h}{2}K_2$. \autoref{construction_procedure} adds edges between distinct pairs of vertices in $B_i$ for $i=1,2,\cdots, r$. Therefore, $Y[V_i']$ is isomorphic to $hK_2$. Also, $X_{2h,r}[V_i^{(2h)}]$ is isomorphic to $hK_2$. Hence, $Y[V_i']$ is isomorphic to $X_{2h,r}[V_i^{(2h)}]$.
		
		In $X_{h,r}$ every vertex in $V_i^{(h)}$ is adjacent to every vertex in $V_j^{(h)}$ for $i\neq j$ and $i,j=0,1,2,\cdots, r$. Similarly, in $X_{2h,r}$, every vertex in $V_i^{(2h)}$ is adjacent to every vertex $V_j^{(2h)}$ for $i\neq j$ and $i,j=0,1,2,\cdots, r$. \autoref{construction_procedure} adds all possible edges between $V_i'$ and $V_j'$ for every pair of distinct indices $i,j\in\{0,1,\ldots,r\}$. Therefore, for $i\neq j$,
		the subgraph $Y\langle V_i',V_j'\rangle$ is isomorphic to
		$X_{2h,r}\langle V_i^{(2h)},V_j^{(2h)}\rangle.$

		Any edge $(u,v)$ in $Y$ is either in $Y[V_i']$ for $i=1,2,\cdots, r$ or in $Y\langle V_i',V_j'\rangle$ for $i\neq j; i,j=0,1,2,\cdots, r$. If $(u,v)\in Y[V_i']$ then $(u,v)\in X_{2h,r}[V_i^{(2h)}]$. If $(u,v)\in Y\langle V_i',V_j' \rangle$ then $(u,v)\in X_{2h,r}\langle V_i^{(2h)},V_j^{(2h)} \rangle$. It holds conversely, that is if $(u,v)\in X_{2h,r}$, then $(u,v)\in Y$. Hence, $Y$ and $X_{2h,r}$ are isomorphic.

		\end{proof}

	\begin{remark}
		\autoref{construction_procedure} constructs $X_{2h,r}$ from $X_{h,r}$ by adding $(r+1)h$ vertices and $\frac{rh}{2}\{3(r+1)h-1\}$. It assists us to construct an infinite family of integral graphs, which is $\{X_{2,r},X_{4,r},\ldots\}$ for $r\geq 1$. A few members of these families are drawn in Appendix.
	\end{remark}
	
	\begin{procedure}[Recursive construction in the $r$-direction]
		\label{proc:r-recursion}
		Let $X_{h,r}=\overline{K}_{h-1}\vee\left(\bigvee_{i=1}^{r}\frac{h}{2}K_2\right).$ Starting from $X_{h,r}$, construct a graph $Y$ as follows.
		
		\begin{enumerate}
			\item \textbf{Vertex Addition:}
			\begin{enumerate}
				\item Add a new set $V_{r+1} = \{x_1,y_1,\ldots,x_{h/2},y_{h/2}\}$ consisting of $h$ new vertices.
			\end{enumerate}

				\item \textbf{Edge Addition:}
				\begin{enumerate}
					\item Add the $\frac{h}{2}$ disjoint edges
					$(x_j,y_j), 1\leq j\leq \frac{h}{2},$ so that
					$Y[V_{r+1}]\cong \frac{h}{2}K_2.$
					
					\item Join every vertex of $V_{r+1}$ to every vertex of
					$X_{h,r}$.
				\end{enumerate}
			\end{enumerate}
	\end{procedure}
	
	\begin{theorem}\label{thm:r-recursion}
		The graph obtained from $X_{h,r}$ by
		\autoref{proc:r-recursion} is isomorphic to $X_{h,r+1}$.
	\end{theorem}
	
	\begin{proof}
		The original graph has the decomposition
		$X_{h,r} = \overline{K}_{h-1} \vee \left(\bigvee_{i=1}^{r}\frac{h}{2}K_2 \right).$
		The newly added set $V_{r+1}$ induces another copy of
		$\frac{h}{2}K_2$, and by construction every vertex of $V_{r+1}$ is adjacent to every vertex of the original graph. Therefore, $Y \cong \overline{K}_{h-1} \vee \left(\bigvee_{i=1}^{r+1}\frac{h}{2}K_2\right)=X_{h,r+1}.$
	\end{proof}
	
	\begin{remark}[Two-directional recursive structure]
		The family $\{X_{h,r}\}$ admits recursive enlargement with respect to both parameters. The vertical direction doubles the parameter $h$, whereas the horizontal direction increases the parameter $r$ by one:
		\[
		\begin{array}{ccc}
			X_{h,r}
			&\xrightarrow{\ r\mapsto r+1\ }&
			X_{h,r+1}
			\\[2mm]
			\Big\downarrow{\scriptstyle h\mapsto 2h}
			&&
			\Big\downarrow{\scriptstyle h\mapsto 2h}
			\\[2mm]
			X_{2h,r}
			&\xrightarrow{\ r\mapsto r+1\ }&
			X_{2h,r+1}.
		\end{array}
		\]
		
		Here vertical arrow represents the transformation
		$X_{h,r}\longmapsto X_{2h,r},$ while horizontal arrow represents $X_{h,r}\longmapsto X_{h,r+1}.$ Consequently, starting from a small member of the family, one can
		generate larger graphs by increasing either parameter independently. 
		
		It should be noted that the horizontal recursion is purely
		graph-theoretic. Although $X_{h,r}$ is defined for every $r\geq1$, a realization by a finite $2$-group requires $r=2^d-1$. One specific example of the two directional recursive structure is shown in \autoref{recursive_example}.
	\end{remark}

	\begin{remark}\label{vertical_tower}
		The procedure \autoref{construction_procedure} keeps the index $2^d$ fixed with $[G:H(G)]=2^d$ while increasing the involution parameter $m$, where $|H(G)|=2^m$. Thus, in the $(m,d)$-parameter space of realizable graphs, direct multiplication by elementary Abelian $2$-groups produces a vertical realization tower $(m,d)\longmapsto(m+t,d).$
	\end{remark}
	
	\begin{theorem}\label{elementary_extension}
		Let $D^{*}(G)\cong X_{h,r}.$ For every integer $t\geq0$, define $G_t=G\times\mathbb Z_2^{\,t}.$ Then $D^{*}(G_t)\cong X_{2^t h,\,r}.$
	\end{theorem}

	\begin{proof}
		By \autoref{exact_characterization}, $\exp(G)=4$ and $H(G)=\{x\in G:x^2=e\}$ is a subgroup satisfying $|H(G)|=h \text{ and } [G:H(G)]=r+1.$ Since every element of $\mathbb Z_2^{\,t}$ has square equal to the
		identity, $H(G_t) = H(G)\times\mathbb Z_2^{\,t}.$ Therefore, $|H(G_t)| = 2^t|H(G)| = 2^t h.$
		Moreover, $[G_t:H(G_t)] = \frac{|G|\,2^t}{|H(G)|\,2^t} = [G:H(G)] = r+1.$ Since $\exp(G)=4$ and $\mathbb Z_2^{\,t}$ has exponent at most $2$, we also have $\exp(G_t)=4.$ Hence \autoref{exact_characterization} gives $D^{*}(G_t) \cong X_{2^t h,\,r}.$
	\end{proof}

	The results of this section show that the parameters $(h,r)$ give several fundamental invariants of the family $X_{h,r}$. 
	 Moreover, elementary abelian direct extensions of realizing groups generate infinite towers of realizable graphs while preserving the second parameter $r$.

 \section{Conclusion}
 In this article, we studied the Proper POE Complement Graph $D^*(G)$ and used it to obtain an infinite family of connected regular adjacency integral graphs $X_{h,r}=\overline{K}_{h-1}\vee\left(\bigvee_{i=1}^{r}\frac{h}{2}K_2\right).$ We determined the adjacency spectrum of $X_{h,r}$ and proved that each member of the family is determined by its adjacency spectrum. For finite $2$-groups, we characterized the regularity of $D^*(G)$ and obtained an exact criterion for when
 $D^*(G)\cong X_{h,r}.$ We further described Abelian and non-Abelian realizations of the graphs $X_{2^m,\,2^d-1}$ and showed that non-isomorphic groups may produce the same graph. Thus, although the adjacency spectrum determines the graph, it does not, in general, determine the underlying group.
 Finally, we presented recursive constructions for the family in both parameters and showed that elementary Abelian direct extensions of realizing groups produce infinite sequence of realizable integral graphs.
 
 \section{Acknowledgment}
 
 During the preparation of this manuscript, the authors used QuillBot in order to improve language and readability. The authors reviewed and edited the content as needed and take full responsibility for the mathematical content and conclusions.

 \phantomsection
 \addcontentsline{toc}{section}{Appendix}
 
 \section*{Appendix}
 
 In this work, we consider the groups of order $2^n$ for different values of $n$. In \autoref{infinite_graph_family}, we define the graphs $X_{h,r}$. The relation between the parameter of the group $n$ and the parameters of the graph $h$ and $r$ as follows: $|G|=2^n=h(r+1),
 \text{ equivalently } r=\frac{2^n}{h}-1$ defined in the \autoref{parameter_restriction} and \autoref{exact_characterization}. In this section, we find the number of graphs in the family of integral graphs $X_{h,r}.$ Also, we calculate the number of groups for different graphs $X_{h,r}.$
 
 Using exhaustive computations in SageMath over all isomorphism classes of groups of orders $2^n$, $3\leq n\leq 8$, we record the number of realizations of the graphs $X_{2,3},\ X_{4,3},\ X_{8,3},\ X_{16,3},\ X_{32,3},
 \ \text{ and }\ X_{64,3}.$ Along with the realization counts, we present some basic graph-theoretic and spectral parameters of these graphs, including their numbers of vertices and edges, independence numbers, and adjacency spectra. Finally, we indicate how the realization sets extend under direct products with elementary abelian $2$-groups, thereby producing infinite towers of realizable connected regular integral graphs.
 
 \begin{definition}\label{realization_fibre}
 	For $m,d\in \mathbb{N}$, define the set of Realization groups of $X_{2^m,\,2^d-1}$ by \[\mathcal{R}_{m,d}
 	=
 	\left\{
 	[G]_{\cong} :
 	\begin{array}{l}
 		G \text{ is a finite $2$-group},\\[2mm]
 		\exp(G)=4,\\[1mm]
 		H(G)=\{x\in G:x^2=e\}\leq G,\\[1mm]
 		|H(G)|=2^m,\\[1mm]
 		[G:H(G)]=2^d
 	\end{array}
 	\right\}.\] Here, $[G]_{\cong}$ denotes the isomorphism class of $G$.
 \end{definition}
 
 \begin{proposition}\label{fibre_characterization}
 	For every pair of positive integers $(m,d)$, $\mathcal{R}_{m,d}
 	=
 	\left\{
 	[G]_{\cong} :
 	D^{*}(G)\cong X_{2^m,\,2^d-1}
 	\right\}.$ Thus, $|\mathcal{R}_{m,d}|$ is precisely the number of
 	non-isomorphic finite $2$-groups realizing the graph
 	$X_{2^m,\,2^d-1}$.
 \end{proposition}
 
 \begin{proof}
 	It follows immediately from \autoref{exact_characterization}.
 \end{proof}

 For $3\leq n\leq8$, exhaustive computation using SageMath \cite{sagemath} over all isomorphism
 classes of groups of order $2^n$ are given in the tables below: 
 \[
 \begin{array}{c|cccccc}
 	n & 3 & 4 & 5 & 6 & 7 & 8\\
 	\hline
 	|G| & 8 & 16 & 32 & 64 & 128 & 256\\
 	\text{Realized graph}
 	& X_{2,3} & X_{4,3} & X_{8,3} & X_{16,3} & X_{32,3} & X_{64,3}
 	\\
 	|V(X_{h,3})|
 	& 7 & 15 & 31 & 63 & 127 & 255
 	\\
 	|E(X_{h,3})|
 	& 21 & 90 & 372 & 1512 & 6096 & 24480
 	\\
 	\alpha(X_{h,3})
 	& 1 & 3 & 7 & 15 & 31 & 63
 	\\
 	$[G:H]$
 	& 4 & 4 & 4 & 4 & 4 & 4
 	\\
 	|\mathcal{R}_{n-2,2}|
 	& 1 & 3 & 7 & 18 & 41 & 90
 	\\
 \end{array}
 \] 
 
 \begin{table}[ht]
 	\centering
 	\small
 	\begin{tabular}{c|c}
 		\hline
 		\text{Graph} & \text{Adjacency spectrum}\\
 		\hline
 		
 		$X_{2,3}$
 		&
 		$\{[6]^1,[-1]^6\}$\\[1mm]
 		
 		$X_{4,3}$
 		&
 		$\{[12]^1,[-3]^3,[1]^3,[0]^2,[-1]^6\}$\\[1mm]
 		
 		$X_{8,3}$
 		&
 		$\{[24]^1,[-7]^3,[1]^9,[0]^6,[-1]^{12}\}$\\[1mm]
 		
 		$X_{16,3}$
 		&
 		$\{[48]^1,[-15]^3,[1]^{21},[0]^{14},[-1]^{24}\}$\\[1mm]
 		
 		$X_{32,3}$
 		&
 		$\{[96]^1,[-31]^3,[1]^{45},[0]^{30},[-1]^{48}\}$\\[1mm]
 		
 		$X_{64,3}$
 		&
 		$\{[192]^1,[-63]^3,[1]^{93},[0]^{62},[-1]^{96}\}$\\
 		
 		\hline
 	\end{tabular}
 \end{table}
 Here, $|G|$ denotes the order of the group $G$,
 $|\mathcal{R}_{n-2,2}|$ denotes the number of isomorphism classes in the corresponding realization set, and $\alpha(X_{h,3})$ denotes the independence number of $X_{h,3}$, that is, the maximum cardinality of an independent set in $X_{h,3}$.

 \begin{remark}\label{parameter_extension}
 	If $D^{*}(G)\cong X_{2^m,\,2^d-1},$ then, for every integer $t\geq0$, $D^{*}\left(G\times\mathbb Z_2^{\,t}\right) \cong
 	X_{2^{m+t},\,2^d-1}.$ In terms of set of Realization groups, $[G]_{\cong}\in\mathcal R_{m,d} \Longrightarrow
 	[G\times\mathbb Z_2^{\,t}]_{\cong}
 	\in\mathcal R_{m+t,d}.$ This can be proved applying \autoref{elementary_extension} with $h=2^m$ and $r=2^d-1$.
 \end{remark}

 \begin{remark}\label{graph_tower}
 	Suppose that $X_{h,r}$ is realizable as $D^{*}(G)$ for some finite $2$-group $G$. Then every graph in the sequence $X_{h,r},\quad X_{2h,r},\quad X_{2^{2}h,r},\quad
 	X_{2^{3}h,r},\quad\ldots$ is realizable as a Proper POE Complement Graph of a finite $2$-group. The realization statement follows directly from
 	\autoref{elementary_extension}. More precisely, $D^{*}\left(G\times\mathbb Z_2^{\,t}\right) \cong X_{2^t h,\,r}$ for every integer $t\geq0$. By \autoref{X_h_r_basic}, $|V(X_{2^t h,r})|= (r+1)2^t h-1.$
 	These orders are distinct for distinct values of $t$, and hence the graphs are pairwise non-isomorphic. Finally, \autoref{X_h_r_basic} and \autoref{Xhr_spectrum} imply that every graph $X_{2^t h,r}$ is connected, regular, and integral. Moreover, these graphs are pairwise non-isomorphic connected regular integral graphs.
 \end{remark}
 
 \begin{remark}
 	Here is the computational results by SageMath for the groups of orders $8$ and $16$.
 	\begin{enumerate}
 		\item For groups of order $8$, computational verification shows that $D^{*}(G)$ is adjacency integral for $4$ of the $5$ groups. Among these, $2$ groups produce graphs belonging to the family $X_{h,r}$ considered in this paper. The remaining $2$ integral graphs are disconnected and do not belong to the family $X_{h,r}$.
 		\item For groups of order $16$, computational verification shows that $D^{*}(G)$ is adjacency integral for $7$ of the $14$ groups. Among these, $5$ groups produce graphs belonging to the family $X_{h,r}$. The remaining $2$ integral graphs are disconnected and lie outside the family $X_{h,r}$.
 		Thus, all connected integral cases of order $16$ are covered by $X_{h,r}$.
 	\end{enumerate}
 \end{remark}

 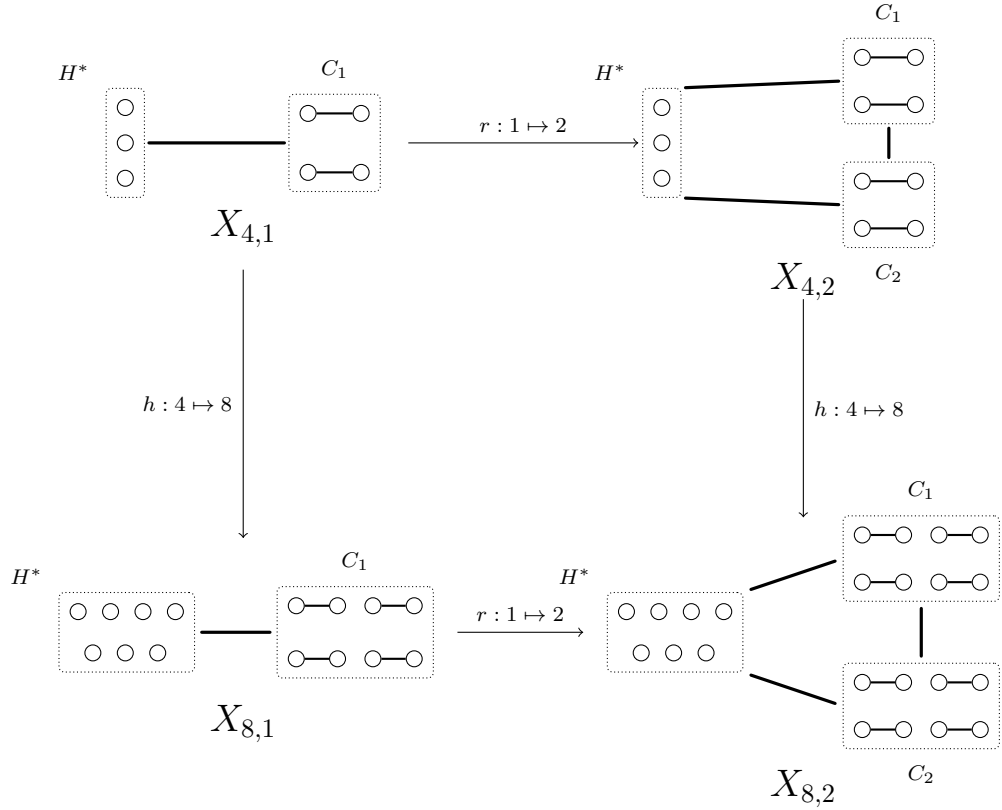
\begin{figure}[ht]
 	\centering
 	
 	\usetikzlibrary{fit,positioning,calc}
 	
 	\begin{tikzpicture}[
 		scale=0.78,
 		every node/.style={font=\scriptsize},
 		v/.style={
 			circle,
 			draw,
 			fill=white,
 			inner sep=0pt,
 			minimum size=6pt
 		},
 		box/.style={
 			draw,
 			rounded corners=2pt,
 			densely dotted,
 			inner sep=4pt
 		},
 		match/.style={thick},
 		joinedge/.style={
 			very thick,
 			line cap=round
 		},
 		arr/.style={->,thin}
 		]
 		
 		\begin{scope}[shift={(0,5.5)}]
 			
 			\node[v] (a1) at (-2.0,0.6) {};
 			\node[v] (a2) at (-2.0,0.0) {};
 			\node[v] (a3) at (-2.0,-0.6) {};
 			
 			\node[box,fit=(a1)(a2)(a3)] (A) {};
 			\node[above left=-1pt and 2pt of A.north west] {$H^*$};
 			
 			\node[v] (b1) at (1.1,0.5) {};
 			\node[v] (b2) at (2.0,0.5) {};
 			\node[v] (b3) at (1.1,-0.5) {};
 			\node[v] (b4) at (2.0,-0.5) {};
 			
 			\draw[match] (b1)--(b2);
 			\draw[match] (b3)--(b4);
 			
 			\node[box,fit=(b1)(b2)(b3)(b4)] (B) {};
 			\node[above=2pt of B.north] {$C_1$};
 			
 			\draw[joinedge,shorten >=1.5pt,shorten <=1.5pt]
 			(A.east)--(B.west);
 			
 			\node at (0,-1.45) {\large $X_{4,1}$};
 			
 		\end{scope}

 		\begin{scope}[shift={(9.5,5.5)}]
 			
 			\node[v] (c1) at (-2.4,0.6) {};
 			\node[v] (c2) at (-2.4,0.0) {};
 			\node[v] (c3) at (-2.4,-0.6) {};
 			
 			\node[box,fit=(c1)(c2)(c3)] (C) {};
 			\node[above left=-1pt and 2pt of C.north west] {$H^*$};
 			
 			\node[v] (d1) at (1.0,1.45) {};
 			\node[v] (d2) at (1.9,1.45) {};
 			\node[v] (d3) at (1.0,0.65) {};
 			\node[v] (d4) at (1.9,0.65) {};
 			
 			\draw[match] (d1)--(d2);
 			\draw[match] (d3)--(d4);
 			
 			\node[box,fit=(d1)(d2)(d3)(d4)] (D) {};
 			\node[above=2pt of D.north] {$C_1$};
 			
 			\node[v] (e1) at (1.0,-0.65) {};
 			\node[v] (e2) at (1.9,-0.65) {};
 			\node[v] (e3) at (1.0,-1.45) {};
 			\node[v] (e4) at (1.9,-1.45) {};
 			
 			\draw[match] (e1)--(e2);
 			\draw[match] (e3)--(e4);
 			
 			\node[box,fit=(e1)(e2)(e3)(e4)] (E) {};
 			\node[below=2pt of E.south] {$C_2$};
 			
 			\draw[joinedge,shorten >=1.5pt,shorten <=1.5pt]
 			(C.north east)--(D.west);
 			
 			\draw[joinedge,shorten >=1.5pt,shorten <=1.5pt]
 			(C.south east)--(E.west);
 			
 			\draw[joinedge,shorten >=1.5pt,shorten <=1.5pt]
 			(D.south)--(E.north);
 			
 			\node at (0,-2.35) {\large $X_{4,2}$};
 			
 		\end{scope}

 		\begin{scope}[shift={(0,-2.8)}]
 			
 			\node[v] (f1) at (-2.8,0.35) {};
 			\node[v] (f2) at (-2.25,0.35) {};
 			\node[v] (f3) at (-1.70,0.35) {};
 			\node[v] (f4) at (-1.15,0.35) {};
 			
 			\node[v] (f5) at (-2.55,-0.35) {};
 			\node[v] (f6) at (-2.00,-0.35) {};
 			\node[v] (f7) at (-1.45,-0.35) {};
 			
 			\node[
 			box,
 			fit=(f1)(f2)(f3)(f4)(f5)(f6)(f7)
 			] (F) {};
 			
 			\node[above left=-1pt and 2pt of F.north west] {$H^*$};
 			
 			\node[v] (g1) at (0.9,0.45) {};
 			\node[v] (g2) at (1.6,0.45) {};
 			\draw[match] (g1)--(g2);
 			
 			\node[v] (g3) at (2.2,0.45) {};
 			\node[v] (g4) at (2.9,0.45) {};
 			\draw[match] (g3)--(g4);
 			
 			\node[v] (g5) at (0.9,-0.45) {};
 			\node[v] (g6) at (1.6,-0.45) {};
 			\draw[match] (g5)--(g6);
 			
 			\node[v] (g7) at (2.2,-0.45) {};
 			\node[v] (g8) at (2.9,-0.45) {};
 			\draw[match] (g7)--(g8);
 			
 			\node[
 			box,
 			fit=(g1)(g2)(g3)(g4)(g5)(g6)(g7)(g8)
 			] (G) {};
 			
 			\node[above=2pt of G.north] {$C_1$};
 			
 			\draw[
 			joinedge,
 			shorten >=2.5pt,
 			shorten <=2.5pt
 			]
 			(F.east)--(G.west);
 			
 			\node at (0,-1.55) {\large $X_{8,1}$};
 			
 		\end{scope}

 		\begin{scope}[shift={(9.5,-2.8)}]
 			
 			\node[v] (h1) at (-3.0,0.35) {};
 			\node[v] (h2) at (-2.45,0.35) {};
 			\node[v] (h3) at (-1.90,0.35) {};
 			\node[v] (h4) at (-1.35,0.35) {};
 			
 			\node[v] (h5) at (-2.75,-0.35) {};
 			\node[v] (h6) at (-2.20,-0.35) {};
 			\node[v] (h7) at (-1.65,-0.35) {};
 			
 			\node[
 			box,
 			fit=(h1)(h2)(h3)(h4)(h5)(h6)(h7)
 			] (H) {};
 			
 			\node[above left=-1pt and 2pt of H.north west] {$H^*$};

 			\node[v] (i1) at (1.0,1.65) {};
 			\node[v] (i2) at (1.7,1.65) {};
 			\draw[match] (i1)--(i2);
 			
 			\node[v] (i3) at (2.3,1.65) {};
 			\node[v] (i4) at (3.0,1.65) {};
 			\draw[match] (i3)--(i4);
 			
 			\node[v] (i5) at (1.0,0.85) {};
 			\node[v] (i6) at (1.7,0.85) {};
 			\draw[match] (i5)--(i6);
 			
 			\node[v] (i7) at (2.3,0.85) {};
 			\node[v] (i8) at (3.0,0.85) {};
 			\draw[match] (i7)--(i8);
 			
 			\node[
 			box,
 			fit=(i1)(i2)(i3)(i4)(i5)(i6)(i7)(i8)
 			] (I) {};
 			
 			\node[above=2pt of I.north] {$C_1$};

 			\node[v] (j1) at (1.0,-0.85) {};
 			\node[v] (j2) at (1.7,-0.85) {};
 			\draw[match] (j1)--(j2);
 			
 			\node[v] (j3) at (2.3,-0.85) {};
 			\node[v] (j4) at (3.0,-0.85) {};
 			\draw[match] (j3)--(j4);
 			
 			\node[v] (j5) at (1.0,-1.65) {};
 			\node[v] (j6) at (1.7,-1.65) {};
 			\draw[match] (j5)--(j6);
 			
 			\node[v] (j7) at (2.3,-1.65) {};
 			\node[v] (j8) at (3.0,-1.65) {};
 			\draw[match] (j7)--(j8);
 			
 			\node[
 			box,
 			fit=(j1)(j2)(j3)(j4)(j5)(j6)(j7)(j8)
 			] (J) {};
 			
 			\node[below=2pt of J.south] {$C_2$};

 			
 			\draw[
 			joinedge,
 			shorten >=3pt,
 			shorten <=3pt
 			]
 			(H.north east)--(I.west);
 			
 			\draw[
 			joinedge,
 			shorten >=3pt,
 			shorten <=3pt
 			]
 			(H.south east)--(J.west);
 			
 			\draw[
 			joinedge,
 			shorten >=2.5pt,
 			shorten <=2.5pt
 			]
 			(I.south)--(J.north);
 			
 			\node at (0,-2.65) {\large $X_{8,2}$};
 			
 		\end{scope}

 		
 		\draw[arr]
 		(2.8,5.5)
 		--
 		node[above] {$r:1\mapsto 2$}
 		(6.7,5.5);
 		
 		\draw[arr]
 		($(G.east)+(0.40,0)$)
 		--
 		node[above] {$r:1\mapsto 2$}
 		($(H.west)+(-0.40,0)$);
 		
 		\draw[arr]
 		(0,3.35)
 		--
 		node[left] {$h:4\mapsto 8$}
 		(0,-1.20);
 		
 		\draw[arr]
 		(9.5,2.85)
 		--
 		node[right] {$h:4\mapsto 8$}
 		(9.5,-0.85);
 		
 	\end{tikzpicture}
 	
 	\caption{A specific example of the two-directional recursive
 		construction. Horizontally, the parameter $r$ increases from
 		$1$ to $2$, while vertically the parameter $h$ increases from
 		$4$ to $8$. A bold line between two dotted boxes represents the
 		complete join between the corresponding parts.}
 	\label{recursive_example}
 \end{figure}

 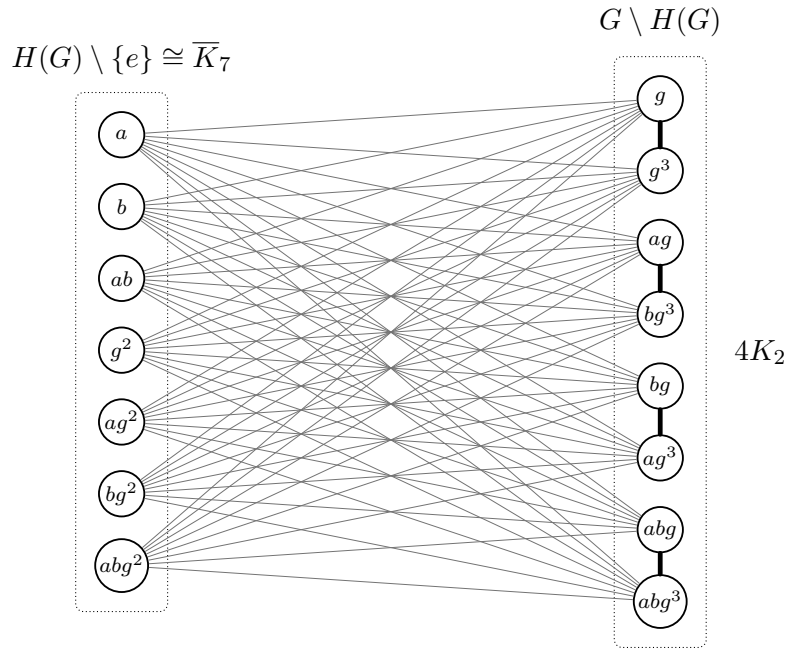
\begin{figure}[ht]
 	\centering
 	
 	\usetikzlibrary{fit,calc}
 	
 	\begin{tikzpicture}[
 		scale=0.95,
 		vertex/.style={
 			circle,
 			draw=black,
 			fill=white,
 			minimum size=6mm,
 			inner sep=0.5pt,
 			font=\scriptsize,
 			line width=0.7pt
 		},
 		crossedge/.style={
 			draw=black!55,
 			line width=0.35pt
 		},
 		matchedge/.style={
 			draw=black,
 			line width=1.8pt,
 			line cap=round,
 			preaction={
 				draw=white,
 				line width=4pt
 			}
 		},
 		partbox/.style={
 			draw=black,
 			densely dotted,
 			rounded corners=3pt,
 			inner sep=7pt
 		}
 		]
 		
 		
 		\node[vertex] (a)    at (-4.5, 3.0) {$a$};
 		\node[vertex] (b)    at (-4.5, 2.0) {$b$};
 		\node[vertex] (ab)   at (-4.5, 1.0) {$ab$};
 		\node[vertex] (g2)   at (-4.5, 0.0) {$g^2$};
 		\node[vertex] (ag2)  at (-4.5,-1.0) {$ag^2$};
 		\node[vertex] (bg2)  at (-4.5,-2.0) {$bg^2$};
 		\node[vertex] (abg2) at (-4.5,-3.0) {$abg^2$};
 		
 		\node[
 		partbox,
 		fit=(a)(b)(ab)(g2)(ag2)(bg2)(abg2)
 		] (Hbox) {};
 		
 		\node[font=\small]
 		at ($(Hbox.north)+(0,0.48)$)
 		{$H(G)\setminus\{e\}\cong\overline K_7$};

 		%
 		
 		\node[vertex] (g)  at (3.0, 3.5) {$g$};
 		\node[vertex] (g3) at (3.0, 2.5) {$g^3$};
 		
 		\node[vertex] (ag)  at (3.0, 1.5) {$ag$};
 		\node[vertex] (bg3) at (3.0, 0.5) {$bg^3$};
 		
 		\node[vertex] (bg)  at (3.0,-0.5) {$bg$};
 		\node[vertex] (ag3) at (3.0,-1.5) {$ag^3$};
 		
 		\node[vertex] (abg)  at (3.0,-2.5) {$abg$};
 		\node[vertex] (abg3) at (3.0,-3.5) {$abg^3$};

 		\node[
 		partbox,
 		fit=(g)(g3)(ag)(bg3)(bg)(ag3)(abg)(abg3)
 		] (Cbox) {};
 		
 		\node[font=\small]
 		at ($(Cbox.north)+(0,0.48)$)
 		{$G\setminus H(G)$};

 		
 		\foreach \u in {a,b,ab,g2,ag2,bg2,abg2}{
 			\foreach \v in {g,g3,ag,bg3,bg,ag3,abg,abg3}{
 				\draw[crossedge] (\u)--(\v);
 			}
 		}

 		
 		\draw[matchedge] (g)--(g3);
 		\draw[matchedge] (ag)--(bg3);
 		\draw[matchedge] (bg)--(ag3);
 		\draw[matchedge] (abg)--(abg3);

 		
 		\node[font=\small]
 		at ($(Cbox.east)+(0.75,0)$)
 		{$4K_2$};
 		
 	\end{tikzpicture}
 	
 	\caption{The Proper POE Complement Graph of
 		$G=\mathbb Z_2^2\rtimes_\theta\mathbb Z_4$.
 		Here,
 		$D^*(G)\cong X_{8,1}
 		=\overline K_7\vee4K_2$.
 		Every vertex of $H(G)\setminus\{e\}$ is adjacent to every vertex
 		of $G\setminus H(G)$, while the four thick vertical edges,
 		arranged one below another, form the matching $4K_2$.}
 	
 	\label{order16-full}
 \end{figure}
 
 
 \newcommand{\XonepicCombined}[2]{%
 	\begin{tikzpicture}[
 		scale=0.52,
 		partbox/.style={
 			draw,
 			rounded corners=2pt,
 			minimum width=1.15cm,
 			minimum height=0.50cm,
 			font=\scriptsize
 		},
 		joinline/.style={very thick}
 		]
 		\node[partbox] (A) at (0,0) {$\overline K_{#1-1}$};
 		\node[partbox] (B) at (1.9,0) {$#2$};
 		\draw[joinline] (A)--(B);
 	\end{tikzpicture}%
 }

 \newcommand{\XtwopicCombined}[2]{%
 	\begin{tikzpicture}[
 		scale=0.50,
 		partbox/.style={
 			draw,
 			rounded corners=2pt,
 			minimum width=1.05cm,
 			minimum height=0.48cm,
 			font=\scriptsize
 		},
 		joinline/.style={very thick}
 		]
 		\node[partbox] (A) at (0,0.9) {$\overline K_{#1-1}$};
 		\node[partbox] (B) at (-1.0,-0.55) {$#2$};
 		\node[partbox] (C) at (1.0,-0.55) {$#2$};
 		
 		\draw[joinline] (A)--(B);
 		\draw[joinline] (A)--(C);
 		\draw[joinline] (B)--(C);
 	\end{tikzpicture}%
 }

 \newcommand{\XthreepicCombined}[2]{%
 	\begin{tikzpicture}[
 		scale=0.48,
 		partbox/.style={
 			draw,
 			rounded corners=2pt,
 			minimum width=1.00cm,
 			minimum height=0.46cm,
 			font=\scriptsize
 		},
 		joinline/.style={very thick}
 		]
 		\node[partbox] (A) at (0,1.0) {$\overline K_{#1-1}$};
 		\node[partbox] (B) at (-1.05,-0.35) {$#2$};
 		\node[partbox] (C) at (1.05,-0.35) {$#2$};
 		\node[partbox] (D) at (0,-1.55) {$#2$};
 		
 		\draw[joinline] (A)--(B);
 		\draw[joinline] (A)--(C);
 		\draw[joinline] (A)--(D);
 		\draw[joinline] (B)--(C);
 		\draw[joinline] (B)--(D);
 		\draw[joinline] (C)--(D);
 	\end{tikzpicture}%
 }

 
 \newcommand{\XonePage}[2]{%
 	\begin{tikzpicture}[
 		scale=0.72,
 		partbox/.style={
 			draw,
 			rounded corners=3pt,
 			minimum width=1.7cm,
 			minimum height=0.72cm,
 			font=\small
 		},
 		joinline/.style={very thick}
 		]
 		\node[partbox] (A) at (0,0) {$\overline K_{#1-1}$};
 		\node[partbox] (B) at (2.8,0) {$#2$};
 		\draw[joinline] (A)--(B);
 	\end{tikzpicture}%
 }
 
 \newcommand{\XtwoPage}[2]{%
 	\begin{tikzpicture}[
 		scale=0.68,
 		partbox/.style={
 			draw,
 			rounded corners=3pt,
 			minimum width=1.45cm,
 			minimum height=0.68cm,
 			font=\small
 		},
 		joinline/.style={very thick}
 		]
 		\node[partbox] (A) at (0,1.15) {$\overline K_{#1-1}$};
 		\node[partbox] (B) at (-1.45,-0.7) {$#2$};
 		\node[partbox] (C) at ( 1.45,-0.7) {$#2$};
 		
 		\draw[joinline] (A)--(B);
 		\draw[joinline] (A)--(C);
 		\draw[joinline] (B)--(C);
 	\end{tikzpicture}%
 }
 
 \newcommand{\XthreePage}[2]{%
 	\begin{tikzpicture}[
 		scale=0.63,
 		partbox/.style={
 			draw,
 			rounded corners=3pt,
 			minimum width=1.35cm,
 			minimum height=0.64cm,
 			font=\small
 		},
 		joinline/.style={very thick}
 		]
 		\node[partbox] (A) at (0,1.2) {$\overline K_{#1-1}$};
 		\node[partbox] (B) at (-1.4,-0.45) {$#2$};
 		\node[partbox] (C) at ( 1.4,-0.45) {$#2$};
 		\node[partbox] (D) at (0,-2.0) {$#2$};
 		
 		\draw[joinline] (A)--(B);
 		\draw[joinline] (A)--(C);
 		\draw[joinline] (A)--(D);
 		\draw[joinline] (B)--(C);
 		\draw[joinline] (B)--(D);
 		\draw[joinline] (C)--(D);
 	\end{tikzpicture}%
 }
 
 
 \newcommand{\XoneFit}[2]{%
 	\begin{tikzpicture}[
 		partbox/.style={
 			draw,
 			rounded corners=3pt,
 			minimum width=1.65cm,
 			minimum height=0.72cm,
 			font=\small
 		},
 		joinline/.style={very thick}
 		]
 		\node[partbox] (A) at (0,0) {$\overline K_{#1-1}$};
 		\node[partbox] (B) at (2.6,0) {$#2$};
 		\draw[joinline] (A)--(B);
 	\end{tikzpicture}%
 }

 \newcommand{\XtwoFit}[2]{%
 	\begin{tikzpicture}[
 		partbox/.style={
 			draw,
 			rounded corners=3pt,
 			minimum width=1.55cm,
 			minimum height=0.70cm,
 			font=\small
 		},
 		joinline/.style={very thick}
 		]
 		\node[partbox] (A) at (0,1.15) {$\overline K_{#1-1}$};
 		\node[partbox] (B) at (-1.35,-0.55) {$#2$};
 		\node[partbox] (C) at ( 1.35,-0.55) {$#2$};
 		
 		\draw[joinline] (A)--(B);
 		\draw[joinline] (A)--(C);
 		\draw[joinline] (B)--(C);
 	\end{tikzpicture}%
 }

 \newcommand{\XthreeFit}[2]{%
 	\begin{tikzpicture}[
 		partbox/.style={
 			draw,
 			rounded corners=3pt,
 			minimum width=1.50cm,
 			minimum height=0.68cm,
 			font=\small
 		},
 		joinline/.style={very thick}
 		]
 		\node[partbox] (A) at (0,1.25) {$\overline K_{#1-1}$};
 		
 		\node[partbox] (B) at (-1.35,-0.35) {$#2$};
 		\node[partbox] (C) at ( 1.35,-0.35) {$#2$};
 		\node[partbox] (D) at (0,-1.75) {$#2$};
 		
 		\draw[joinline] (A)--(B);
 		\draw[joinline] (A)--(C);
 		\draw[joinline] (A)--(D);
 		\draw[joinline] (B)--(C);
 		\draw[joinline] (B)--(D);
 		\draw[joinline] (C)--(D);
 	\end{tikzpicture}%
 }

 
 \tikzset{
 	Xpart/.style={
 		draw=black,
 		rounded corners=2pt,
 		minimum width=1.35cm,
 		minimum height=0.58cm,
 		font=\scriptsize,
 		fill=white
 	},
 	Xjoin/.style={
 		draw=black,
 		line width=1.2pt
 	}
 }
 
 \newcommand{\XoneFullPage}[2]{%
 	\begin{tikzpicture}[baseline=(current bounding box.center)]
 		
 		\node[Xpart] (A) at (0,0)
 		{$\overline K_{#1-1}$};
 		
 		\node[Xpart] (B) at (2.15,0)
 		{$#2$};
 		
 		\draw[Xjoin] (A)--(B);
 		
 	\end{tikzpicture}%
 }
 
 \newcommand{\XtwoFullPage}[2]{%
 	\begin{tikzpicture}[baseline=(current bounding box.center)]
 		
 		\node[Xpart] (A) at (0,1.05)
 		{$\overline K_{#1-1}$};
 		
 		\node[Xpart] (B) at (-1.15,-0.65)
 		{$#2$};
 		
 		\node[Xpart] (C) at (1.15,-0.65)
 		{$#2$};
 		
 		\draw[Xjoin] (A)--(B);
 		\draw[Xjoin] (A)--(C);
 		\draw[Xjoin] (B)--(C);
 		
 	\end{tikzpicture}%
 }
 
 \newcommand{\XthreeFullPage}[2]{%
 	\begin{tikzpicture}[baseline=(current bounding box.center)]
 		
 		\node[Xpart] (A) at (0,1.25)
 		{$\overline K_{#1-1}$};
 		
 		\node[Xpart] (B) at (-1.25,0)
 		{$#2$};
 		
 		\node[Xpart] (C) at (1.25,0)
 		{$#2$};
 		
 		\node[Xpart] (D) at (0,-1.35)
 		{$#2$};
 		
 		\draw[Xjoin] (A)--(B);
 		\draw[Xjoin] (A)--(C);
 		\draw[Xjoin] (A)--(D);
 		\draw[Xjoin] (B)--(C);
 		\draw[Xjoin] (B)--(D);
 		\draw[Xjoin] (C)--(D);
 		
 	\end{tikzpicture}%
 }
 
 \newcommand{\XfourFullPage}[2]{%
 	\begin{tikzpicture}[baseline=(current bounding box.center)]
 		
 		\node[Xpart] (A) at (0,1.60)
 		{$\overline K_{#1-1}$};
 		
 		\node[Xpart] (B) at (-1.55,0.45)
 		{$#2$};
 		
 		\node[Xpart] (C) at (-0.95,-1.25)
 		{$#2$};
 		
 		\node[Xpart] (D) at (0.95,-1.25)
 		{$#2$};
 		
 		\node[Xpart] (E) at (1.55,0.45)
 		{$#2$};
 		
 		
 		\draw[Xjoin] (A)--(B);
 		\draw[Xjoin] (A)--(C);
 		\draw[Xjoin] (A)--(D);
 		\draw[Xjoin] (A)--(E);
 		
 		\draw[Xjoin] (B)--(C);
 		\draw[Xjoin] (B)--(D);
 		\draw[Xjoin] (B)--(E);
 		
 		\draw[Xjoin] (C)--(D);
 		\draw[Xjoin] (C)--(E);
 		
 		\draw[Xjoin] (D)--(E);
 		
 	\end{tikzpicture}%
 }

 \begin{figure}[p]
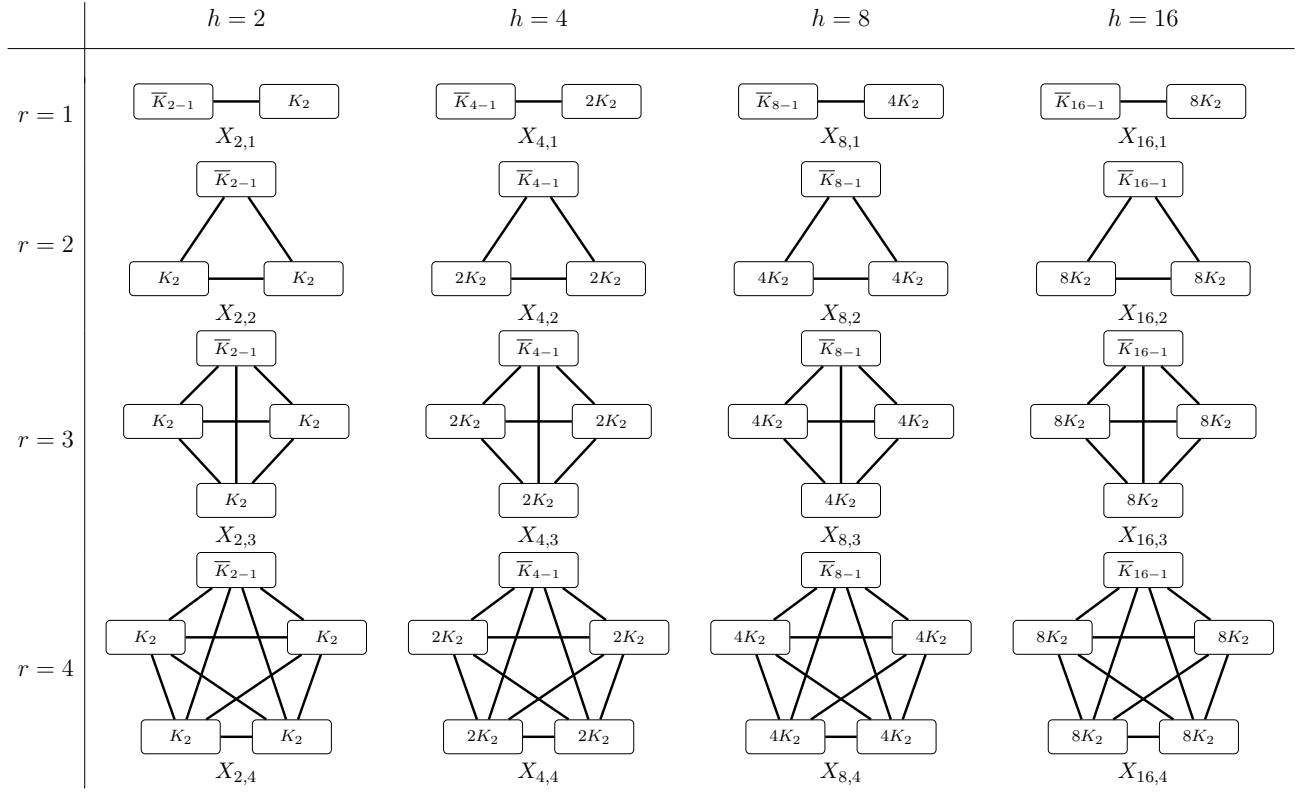

 	\centering
 	
 	\begin{adjustbox}{
 			max totalsize={0.98\textwidth}{0.84\textheight},
 			center
 		}
 		
 		\setlength{\tabcolsep}{5pt}
 		\renewcommand{\arraystretch}{1.15}
 		
 		\begin{tabular}{c|c c c c}
 			
 			&
 			\textbf{$h=2$}
 			&
 			\textbf{$h=4$}
 			&
 			\textbf{$h=8$}
 			&
 			\textbf{$h=16$}
 			\\[2mm]
 			
 			\hline
 			\\[-1mm]
 			
 			
 			\textbf{$r=1$}
 			&
 			\begin{tabular}{c}
 				\XoneFullPage{2}{K_2}\\[1mm]
 				{\small $X_{2,1}$}
 			\end{tabular}
 			&
 			\begin{tabular}{c}
 				\XoneFullPage{4}{2K_2}\\[1mm]
 				{\small $X_{4,1}$}
 			\end{tabular}
 			&
 			\begin{tabular}{c}
 				\XoneFullPage{8}{4K_2}\\[1mm]
 				{\small $X_{8,1}$}
 			\end{tabular}
 			&
 			\begin{tabular}{c}
 				\XoneFullPage{16}{8K_2}\\[1mm]
 				{\small $X_{16,1}$}
 			\end{tabular}
 			
 			\\[5mm]
 			
 			
 			\textbf{$r=2$}
 			&
 			\begin{tabular}{c}
 				\XtwoFullPage{2}{K_2}\\[1mm]
 				{\small $X_{2,2}$}
 			\end{tabular}
 			&
 			\begin{tabular}{c}
 				\XtwoFullPage{4}{2K_2}\\[1mm]
 				{\small $X_{4,2}$}
 			\end{tabular}
 			&
 			\begin{tabular}{c}
 				\XtwoFullPage{8}{4K_2}\\[1mm]
 				{\small $X_{8,2}$}
 			\end{tabular}
 			&
 			\begin{tabular}{c}
 				\XtwoFullPage{16}{8K_2}\\[1mm]
 				{\small $X_{16,2}$}
 			\end{tabular}
 			
 			\\[5mm]
 			
 			
 			\textbf{$r=3$}
 			&
 			\begin{tabular}{c}
 				\XthreeFullPage{2}{K_2}\\[1mm]
 				{\small $X_{2,3}$}
 			\end{tabular}
 			&
 			\begin{tabular}{c}
 				\XthreeFullPage{4}{2K_2}\\[1mm]
 				{\small $X_{4,3}$}
 			\end{tabular}
 			&
 			\begin{tabular}{c}
 				\XthreeFullPage{8}{4K_2}\\[1mm]
 				{\small $X_{8,3}$}
 			\end{tabular}
 			&
 			\begin{tabular}{c}
 				\XthreeFullPage{16}{8K_2}\\[1mm]
 				{\small $X_{16,3}$}
 			\end{tabular}
 			
 			\\[5mm]
 			
 			
 			\textbf{$r=4$}
 			&
 			\begin{tabular}{c}
 				\XfourFullPage{2}{K_2}\\[1mm]
 				{\small $X_{2,4}$}
 			\end{tabular}
 			&
 			\begin{tabular}{c}
 				\XfourFullPage{4}{2K_2}\\[1mm]
 				{\small $X_{4,4}$}
 			\end{tabular}
 			&
 			\begin{tabular}{c}
 				\XfourFullPage{8}{4K_2}\\[1mm]
 				{\small $X_{8,4}$}
 			\end{tabular}
 			&
 			\begin{tabular}{c}
 				\XfourFullPage{16}{8K_2}\\[1mm]
 				{\small $X_{16,4}$}
 			\end{tabular}
 			
 		\end{tabular}
 		
 	\end{adjustbox}
 	
 	\vspace{3mm}
 	
 	\caption{Examples from the scalable family $X_{h,r}$ for
 		$h\in\{2,4,8,16\}$ and $r\in\{1,2,3,4\}$.
 		The columns correspond to fixed values of $h$, while the rows
 		correspond to fixed values of $r$. Each box represents one
 		part of the graph, and every thick line between two boxes
 		indicates that the corresponding parts are completely joined.}
 	
 	\label{all-scalable-family}
 	
 \end{figure}

\end{document}